\documentclass[11pt,reqno]{amsart}
\usepackage[dvipsnames,usenames]{color}
\usepackage{lipsum}
\usepackage{enumerate}
\usepackage{filecontents}
\usepackage[dvipsnames]{xcolor}
\usepackage{hyperref}
\usepackage{cleveref}
\usepackage{epsfig}
\usepackage{epstopdf}
\newcommand\myshade{85}
\colorlet{mylinkcolor}{NavyBlue}
\colorlet{mycitecolor}{NavyBlue}
\colorlet{myurlcolor}{NavyBlue}

\hypersetup{
  linkcolor  = mylinkcolor!\myshade!black,
  citecolor  = mycitecolor!\myshade!black,
  urlcolor   = myurlcolor!\myshade!black,
  colorlinks = true,
}
\usepackage{xypic}
\usepackage{tabularx,booktabs}
\usepackage{amssymb,amsmath,amsthm}
\usepackage{graphicx,color}
\usepackage{hyperref}
\usepackage{mathrsfs}
\usepackage{adjustbox}
\usepackage{caption}
\renewcommand{\phi}{\varphi}
\newtheorem{theorem}{Theorem}[section]
\newtheorem{lemma}[theorem]{Lemma}
\newtheorem{cor}[theorem]{Corollary}
\newtheorem{conj}[theorem]{Conjecture}
\newtheorem{prop}[theorem]{Proposition}

\theoremstyle{definition}
\newtheorem{definition}[theorem]{Definition}
\newtheorem{rmk}[theorem]{Remark}

\renewcommand{\phi}{\varphi}
\newcommand{\Z}{\mathbb{Z}}
\newcommand{\Q}{\mathbb{Q}}
\newcommand{\C}{\mathrm{c}}
\newcommand{\Spin}{\operatorname{Spin}}
\newcommand{\braket}[2]{\langle #1, #2 \rangle}
\newcommand{\into}{\hookrightarrow}
\newcommand{\Sh}{\operatorname{Short}}
\newcommand{\T}{\mathcal{T}}
\newcommand{\Ch}{\operatorname{Char}}
\newcommand{\supp}{\operatorname{supp}}

\author[W. Ballinger]{William Ballinger}
\address {Department of Mathematics, Princeton University, Princeton, NJ 08544}
\email{wballinger@princeton.edu}

\author[Y. Ni]{Yi Ni}
\address {Department of Mathematics, California Institute of Technology, Pasadena, CA 91125}
\email{yini@caltech.edu}

\author[T. Ochse]{Tynan Ochse}
\address {Department of Mathematics, University of Texas, Austin, TX 78712}
\email{tochse@utexas.edu}

\author[F. Vafaee]{Faramarz Vafaee}
\address {Department of Mathematics, Duke University, Durham, NC 27708}
\email{vafaee@math.duke.edu}

\begin{document}
\title{The prism manifold realization problem III}

\date{}

\maketitle

\begin{abstract}
Every prism manifold can be parametrized by a pair of relatively prime integers $p>1$ and $q$. In our earlier papers, we determined a complete list of prism manifolds $P(p, q)$ that can be realized by positive integral surgeries on knots in $S^3$ when $q<0$ or $q>p$; in the present work, we solve the case when $0<q<p$. This completes the solution of the realization problem for prism manifolds.
\end{abstract}

\section{Introduction}\label{sec:Introduction}
Let $P(p,q)$ be an oriented prism manifold with Seifert invariants
\[
(-1; (2,1), (2,1), (p,q)),
\]
where $q$ and $p>1$ are relatively prime integers. In~\cite{Prism2016, Prism2017}, we solved the Dehn surgery realization problem of prism manifolds for $q<0$ and for $q>p$. The theme of the present work is to settle the remaining case $0<q<p$. In~\cite[Tables~1~and~2]{Prism2016}, the authors give a tabulation of prism manifolds that can be obtained by positive integral Dehn surgery on {\it Berge--Kang knots}~\cite{BergeKang}. The tables conjecturally account for all realizable prism manifolds; in particular,~\cite[Table~2]{Prism2016} suggests that for a realizable $P(p,q)$ with $q>0$, we must have $p\le 2q+1$. Indeed, this is the case:
\begin{theorem}\label{thm:Pbound}
If $P(p,q)$ with $q>0$ can be obtained by surgery on a knot $K\subset S^3$, then $p\le2q+1$. If $p=2q+1$, then $K$ is the torus knot $T(2q+1,2)$.
\end{theorem}

Doig, in~\cite[Conjecture~12]{Doig2}, conjectured that if $P(p,q)$ is realizable, then $p\le 2|q|+1$. The main result of~\cite{Prism2016} settles the conjecture for $q<0$; Theorem~\ref{thm:Pbound} verifies it for $q>0$.

Our second main result, Theorem~\ref{thm:Realization} below, provides the solution of the realization problem for those $P(p,q)$ with $q<p<2q$.
\begin{theorem}\label{thm:Realization}
The prism manifold $P(p,q)$ with $q<p<2q$ can be obtained by $4q$--surgery on a knot $K\subset S^3$ if and only if $q=\frac{1}{r^2-2r-1}(r^2p-1)$, with $r\le -3$ odd and $p\equiv -2r+5 \pmod{r^2-2r-1}$. Moreover, in this case, there exists a Berge--Kang knot $K_0$ such that $P(p,q)\cong S^3_{4q}(K_0)$, and that $K$ and $K_0$ have isomorphic knot Floer homology groups.
\end{theorem}

\begin{rmk}
If we allow $r=-1$ in Theorem~\ref{thm:Realization}, we get $p=2q+1$: see Theorem~\ref{thm:Pbound}.
\end{rmk}

\subsection{The spherical manifold realization problem} The spherical manifold realization problem asks which spherical manifolds arise from positive integral surgery along a knot in $S^3$. Theorems~\ref{thm:Pbound} and~\ref{thm:Realization} and our earlier results~\cite{Prism2016,Prism2017}, combined with Gu's work~\cite{Gu2014} and Greene's work~\cite{greene:LSRP}, provide a complete classification of realizable spherical manifolds. The interest is in finding a complete classification of knots in $S^3$ on which Dehn surgery produce spherical manifolds. In~\cite{Berge}, Berge proposed
a complete list of knots in $S^3$ with lens space surgeries. Indeed, Berge's conjecture states that the {\it P/P knots} form a complete list of knots in $S^3$ that admit lens space surgeries. All the known examples of knots on which surgeries will result in non-lens space spherical manifolds are {\it P/SF knots}. We repeat the following conjecture from~\cite[Conjecture~1.7]{Prism2016}: it is a generalization of Berge's conjecture.

{\conj Let $K$ be a knot in $S^3$ that admits an integral surgery to a spherical manifold. Then $K$ is either a $P/SF$ or a $P/P$ knot.}

\subsection{Methodology}\label{methodology} We first provide a brief overview of the methodology undertaken to solve the prism manifold realization problem in the cases $q<0$ and $q>p$: the proof in both cases draws inspiration from that of Greene for lens spaces~\cite{greene:LSRP}. We then discuss how (and why) the methodology is modified for the case of the present work.

\noindent We first require a combinatorial definition.

\begin{definition}\label{defn:changemaker}
A vector $\sigma=(\sigma_0,\sigma_1,\dots,\sigma_{n+1})\in\mathbb Z^{n+2}$ that satisfies $0\le\sigma_0\le\sigma_1\le\cdots\le\sigma_{n+1}$ is a {\it changemaker vector} if for every $k$, with $0\le k\le\sigma_0+\sigma_1+\cdots+\sigma_{n+1}$, there exists a subset $S\subset\{0,1,\dots,n+1\}$
such that $k=\sum_{i\in S}\sigma_i$.
\end{definition}

The key idea is to use the {\it correction terms} in Heegaard Floer homology in tandem with Donaldson's Theorem~A. The following is immediate from~\cite[Theorem~3.3]{greene:LSRP}.

\begin{theorem}\label{changemakerlatticeembedding}
Suppose that $P(p,q)$ bounds a sharp four-manifold $X(p,q)$. If $P(p, q)$ arises from positive integer surgery on a knot $K$ in $S^3$, then the intersection lattice on $X(p,q)$ embeds as the orthogonal complement $\sigma^\perp$ of some changemaker vector $\sigma \in \Z^{n+2}$, with $n+1=b_2(X)$.
\end{theorem}

See Section~\ref{sect:SharpCob} for the definition of a {\it sharp} four-manifold,
and see Subsection~\ref{notation} for the definition of the {\it intersection lattice}. When $q<0$ or $q>p$, it turns out that $P(p,q)$ bounds a sharp four-manifold $X(p,q)$. We then solved a combinatorial problem: we classified all lattices isomorphic to the intersection lattice of $X(p,q)$, whose complements are changemakers in $\Z^{n+2}$. There is a heavy analysis of lattices involved that forms the main body of~\cite{Prism2016,Prism2017}. Finally, we verified that for every $(p,q)$ corresponding to such a lattice, $P(p,q)$
is indeed realized by surgery on a P/SF knot.

\begin{figure}[t]
\includegraphics[scale=.25]{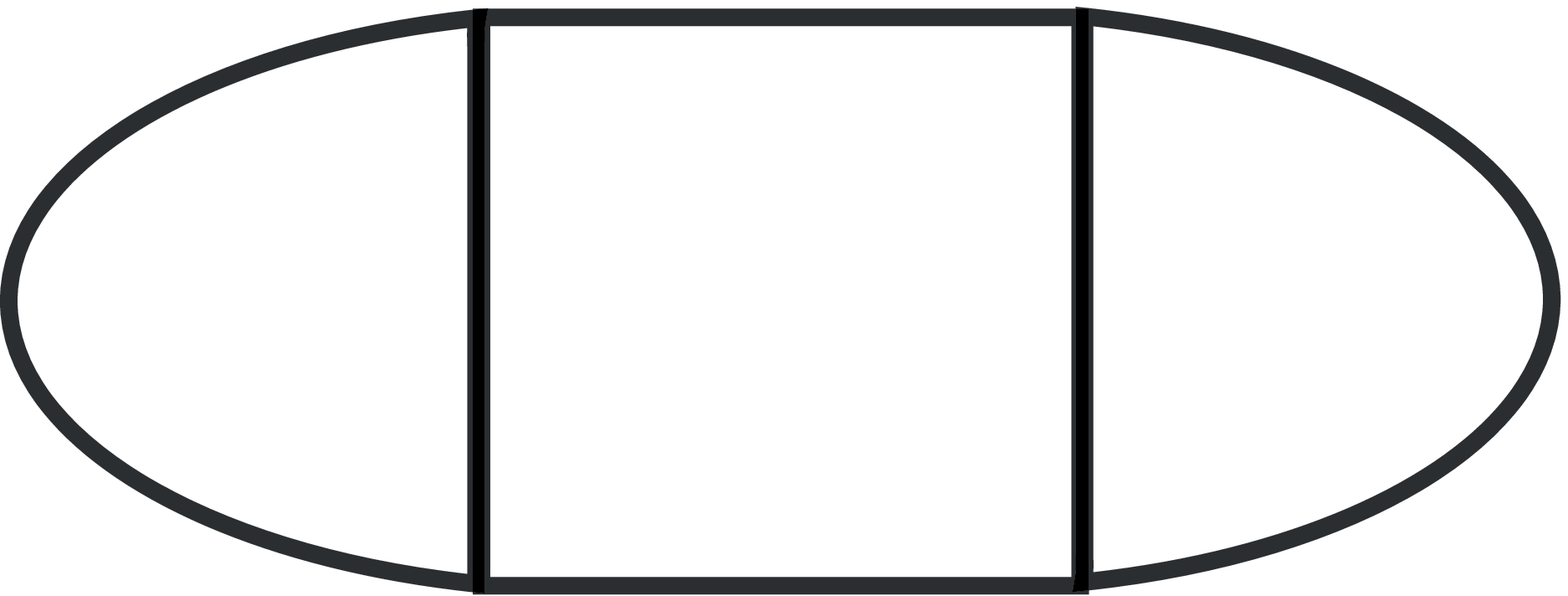}
\put(-72,20){$W$}
\put(-110,55){$P(2,1)$}
\put(-114,20){$Z_2$}
\put(-60,55){$P(p,q)$}
\put(-34,20){$-W_{4q}$}

\caption{Schematic picture of the closed four--manifold $\widehat X= Z_2\cup W\cup -W_{4q}$. We have $X=W\cup_{P(p,q)} -W_{4q}$, $Z=Z_2\cup_{P(2,1)} W$.}
\label{donaldson}
\end{figure}

We now turn our attention to the case $0<q<p$. In light of Theorem~\ref{thm:Pbound}, it suffices to consider $q<p<2q$. When $q<p<2q$, $P(p,q)$ does not bound a sharp four--manifold. Thus, we cannot use the embedding restriction of Theorem~\ref{changemakerlatticeembedding} -- an essential to the classification of realizable prism manifolds in the previous two cases. Our strategy to prove Theorem~\ref{thm:Realization} is to replace Theorem~\ref{changemakerlatticeembedding} with another lattice theoretic obstruction for $P(p,q)$ to being realizable, as follows. The prism manifold $P(2,1)$ bounds a rational homology four-ball $Z_2$ (the left two components of Figure~\ref{fig:Kirby3} where the $0$--framed unknot is replaced by a dotted circle and $a_{-1}=2$); and that there exists a negative definite cobordism $W$ from $P(2,1)$ to $P(p,q)$ (the right $n+1$ components of Figure~\ref{fig:Kirby3}). Suppose that $P(p,q)$ arises from surgery on a knot $K\subset S^3$, and let $W_{4q}=W_{4q}(K)$ be the corresponding two-handle cobordism obtained by attaching a two-handle to the four-ball along the knot $K$ with framing $4q$. Form $Z:=Z_2\cup_{P(2,1)}W$; it will be a smooth four-manifold with boundary $P(p,q)$. The intersection lattice on $Z$ is $\Lambda(q,-p)$, which is defined in~Definition~\ref{def:CType}. Form $X:=W\cup(-W_{4q})$. We prove that the intersection lattice on $X$ is isomorphic to $D_4\oplus \Z^{n-2}$. Finally, form $\widehat X:= Z\cup (-W_{4q})$; see Figure~\ref{donaldson}. It follows that $\widehat X$ is a smooth, closed, simply connected, negative definite four-manifold with $b_2(Z)=n+2$ for some $n\ge0$. Now, Donaldson's Theorem~A~\cite{Donaldson1983} implies that the intersection lattice on $\widehat X$ is the Euclidean integer lattice $\Z^{n+2}$. This provides a necessary condition for $P(p,q)$ to be realizable: the lattice $\Lambda(q,-p)$ embeds as a codimension one sublattice of $\Z^{n+2}$. Our new obstruction now reads as follows:

\begin{theorem}\label{thm:restriction}
Suppose $P(p, q)$ with $q<p<2q$ arises from positive integer surgery on a knot $K$ in $S^3$.
\begin{itemize}
\item[(a)] The linear lattice $\Lambda(q,-p)$ embeds as the orthogonal complement to a changemaker $\sigma\in \Z^{n+2}, n+1=b_2(Z)$.\\
\item[(b)] There is an embedding of $D_4 \oplus \Z^{n-2}$ into $\Z^{n+2}$ such that there exists some short characteristic covector $\chi$ for $D_4 \oplus \Z^{n-2}$ with $\braket{\chi}{\sigma} = i$ if and only if $-2q + g(K) \le i \le 2q-g(K)$.
\end{itemize}
\end{theorem}

The strategy is now apparent: determine the list of all pairs $(p,q)$ which pass the embedding restriction of Theorem~\ref{thm:restriction}. Finally, we verify that every manifold in our list is indeed realized by a knot surgery: we do so by comparing the list with the list of realizable manifolds tabulated in~\cite[Table~2]{Prism2016}. It must be noted that Part~(a) of Theorem~\ref{thm:restriction} only provides a necessary condition for the prism manifold $P(p,q)$ to be realizable. Indeed, it is easy to find pairs $(p,q)$ that satisfy Part~(a) of Theorem~\ref{thm:restriction}, but the corresponding prism manifolds are not realizable; for example $P(13,9)$ and $P(16,9)$. The $9$--surgery on the torus knot $T(2,5)$ is $L(9,13)\cong L(9,16)$, then work of Greene~\cite{greene:LSRP} shows that the corresponding linear lattice satisfies Part~(a)~of Theorem~\ref{thm:restriction}. However, the manifold $P(16,9)$ is not realizable because of the parity of $16$ ($p$ is always odd for a realizable $P(p,q)$~\cite{Prism2016}); and neither is $P(13,9)$ by Theorem~\ref{thm:Realization}.

In the previous cases $q<0$ and $q>p$ as well as in the lens space realization problem~\cite{greene:LSRP}, the first step was finding a sharp four-manifold bounded by $P(p,q)$ (respectively, the lens space $L(p,q)$): in each case a negative definite four-manifold was found; then it was almost immediate from the previous works of Ozsv\'ath and Szab\'o~\cite{OSzBrDoub, OSzPlumbed} that the four-manifold is sharp. For the case at hand, however, $P(p,q)$ does not bound a sharp four-manifold.
We need to carefully analyze the {\it $\mathrm d$--invariants} of $P(p,q)$ in each Spin$^c$ structure in terms of the $\mathrm d$--invariants of certain Spin$^c$ structures of $P(2,1)$ and the grading shift of the cobordism $W$. In particular, we generalize the notion of sharpness to cobordisms between rational homology spheres, and show that the cobordism $W$ is sharp (Proposition~\ref{prop:SharpCobordism}): again, see Figure~\ref{donaldson}. Using that the intersection lattice on $X$ is isomorphic to $D^4\oplus \Z^{n-2}$, it will be immediate that $X$ is a sharp four--manifold (Corollary~\ref{sharp}). Using this finding, we are able to prove Theorem~\ref{thm:restriction} and translate it into a more practical condition on the changemaker vector~$\sigma$~(Proposition~\ref{prop:T0T1}).

\subsection{Notations}\label{notation} We use homology groups with integer coefficients throughout the paper. For a compact four--manifold $X$, regard $H_2(X)$ as an inner product space equipped with the intersection pairing $Q_X$ on $X$. Also, we refer to $(H_2(X), -Q_X)$ as the {\it intersection lattice} on $X$, where $-Q_X$ denotes the negation of the pairing of $Q_X$. Finally, we call an oriented three--manifold $Y$ a {\it realizable manifold} if it can be obtained by positive integral surgery on a knot in $S^3$ .

\subsection{Organization} This paper is organized as follows. In Section~\ref{sec:p2q}, we prove Theorem~\ref{thm:Pbound}, thus solve the case of the realization problem when $2q<p$. In Section~\ref{sec:latticeinput}, we collect some basic results about linear lattices and changemaker lattices from \cite{greene:LSRP}. In Section~\ref{sec:cobordism}, we study the topology of a certain type of cobordism between rational homology $3$--spheres.
In Section~\ref{sect:SharpCob}, we define sharp cobordisms, and prove that the cobordism $W$ between $P(2,1)$ and $P(p,q)$ is sharp. In Section~\ref{sec:changemaker}, we use the result in Section~\ref{sect:SharpCob} to prove a strengthened changemaker condition in the case $q<p<2q$. In
Section~\ref{sec:dbounding} and Section~\ref{sec:d=0}, we use the strengthened changemaker condition to enumerate all the possible changemaker lattices we can have.
In Section~\ref{sec:pandq}, we determine the pairs $(p,q)$ corresponding to the changemaker lattices, thus finish the proof of Theorem~\ref{thm:Realization}.

\subsection*{Acknowledgements} This project started during Caltech's Summer Undergraduate Research Fellowships (SURF) program in the summer of 2017.
Y.~N. was partially supported by NSF grant numbers DMS-1252992
and DMS-1811900. F.~V. was partially supported by an AMS-Simons Travel Grant.

\section{Proof of Theorem~\ref{thm:Pbound}}\label{sec:p2q}

The goal of this section is to prove the following upper bound of $p$, and then to prove Theorem~\ref{thm:Pbound}. Recall that we assume $q>0$. 

\begin{prop}\label{prop:Pbound}
If $P(p,q)$ is realizable, then $p\le2q+1$.
\end{prop}

\begin{rmk}
If $P(p,q)$ is realizable with $p=2|q|\pm1$, then $K$ must be a torus knot~\cite[Theorem~1.6]{NiZhang}. Recall that for a realizable $P(p,q)$, $p$ is odd~\cite{Prism2016}. In particular, if we restrict attention to hyperbolic knots on which surgeries will result in $P(p,q)$, then $p\le 2|q|-3$.
\end{rmk}


\subsection{The Casson--Walker invariant of $P(p,q)$}

Let
\begin{equation}\label{eq:AlexanderPolynomial}
\Delta_K(T)= \alpha_0 + \sum_{i>0}\alpha_i(T^i+T^{-i})
\end{equation}
be the normalized Alexander polynomial of $K$. If $K$ admits an L-space surgery, then $|\alpha_i|\le1$, $\alpha_{g(K)}=1$, and $+1$ and $-1$ appear alternatingly among the nonzero $\alpha_i$~\cite[Theorem~1.2]{OSzLens}.

Given a real number $x$, let $\{x\}=x-\lfloor x\rfloor$ be the fractional part of $x$. Given a pair of coprime integers $n,m$ with $n>0$, let $\mathbf s(m,n)$ be the Dedekind sum
\[
\mathbf s(m,n)=\sum_{i=1}^{n-1}\left(\left(\frac in\right)\right)\left(\left(\frac{im}n\right)\right),
\]
where
\[
((x))=\left\{
\begin{array}{ll}
\{x\}-\frac12, &\text{if }x\in\mathbb R\setminus\mathbb Z,\\
0, &\text{if }x\in\mathbb Z.
\end{array}
\right.
\]

Let $\lambda(\cdot)$ be the Casson--Walker invariant \cite{Walker}, normalized so that \[\lambda(S^3_1(T(3,2)))=2.\]
By \cite[Proposition~6.1.1]{Lescop}, the Casson--Walker invariant of $P(p,q)$ can be computed by the formula
\begin{equation*}\label{eq:CW}
\lambda(P(p,q))=\frac1{12}\left(-\frac pq(\frac1{p^2}-\frac12)-\frac qp+3+12\mathbf s(q,p)\right).
\end{equation*}
Since the Dedekind sum satisfies the reciprocity law
\[
\mathbf s(q,p)+\mathbf s(p,q)=\frac1{12}(\frac pq+\frac qp+\frac1{pq})-\frac14,
\]
we get
\begin{equation}\label{eq:CWprism}
\lambda(P(p,q))=\frac p{8q}-\mathbf s(p,q).
\end{equation}

On the other hand, the surgery formula for the Casson--Walker invariant \cite[Theorem 2.8]{BL} implies that
\begin{eqnarray}
\lambda(S^3_{4q}(K))&=&-\mathbf s(1,4q)+\frac1{4q}\Delta''_K(1)\nonumber\\
&=&-\frac{(2q-1)(4q-1)}{24q}+\frac1{4q}\Delta''_K(1).\label{eq:CWsurg}
\end{eqnarray}

\begin{lemma}\label{lem:mod4}
For realizable $P(p,q)$ with $q$ odd, $p\equiv-1\pmod4$.
\end{lemma}
\begin{proof}
By combining~\eqref{eq:CWprism}~and~\eqref{eq:CWsurg}, we have
\begin{eqnarray*}
&&-\frac{(2q-1)(4q-1)}{24q}+\frac1{4q}\Delta''_K(1) \\
&=&\lambda(P(p,q))\\
&\equiv&\frac p{8q}-\sum_{i=1}^{q-1}(\frac{i}q-\frac12)(\frac{pi}q-\frac12)\pmod1\\
&=&\frac p{8q}-\frac{p(q-1)(2q-1)}{6q}+\frac{p(q-1)}{4}.
\end{eqnarray*}
Multiplying both sides by $24q$, we get
\[
1 - 6 q + 8 q^2 + p (-1 + 6 q - 2 q^2)\equiv 6\Delta''_K(1) \pmod{24q}.
\]
Since $\Delta''_K(1)$ is even and $p,q$ are odd, we get
\[
2q+1+p(2q+1)\equiv0\pmod4.
\]
So $p\equiv-1\pmod4$.
\end{proof}

\subsection{The Spin$^c$ structures}

The $i$-th {\it torsion coefficient} of a knot $K$ is defined to be
\[
t_i(K)= \sum_{j\ge 1}j\alpha_{i+j},
\]
for $i\ge0$, where the $\alpha_i$ are as in~\eqref{eq:AlexanderPolynomial}.
Let
\[\varepsilon_i=t_i-t_{i+1}.\]
When $K$ admits an L-space surgery, it is proved in \cite[Proposition~7.6]{RasThesis} that
\begin{equation*}\label{eq:tMono}
\varepsilon_i\in\{0,1\}.
\end{equation*}
Suppose $4q$--surgery on $K$ is $P(p,q)$, then $4q\ge2g(K)-1$ \cite{OSzRatSurg}. So
\begin{equation}\label{eq:genus2q}
g(K)\le2q.
\end{equation}
Since $a_{g(K)}=1$ and $a_i=0$ when $i>g(K)$, it follows from the definition of $t_i$ that
\begin{equation}\label{eq:ti=0}
t_i=0\quad \text{if and only if }i\ge g(K).
\end{equation}
In particular, by
 (\ref{eq:genus2q}), we get
\begin{equation}\label{eq:t2q}
t_{2q}=0.
\end{equation}

For $i>0$,
\begin{align*}\label{eq:bi}
\alpha_i&=t_{i-1}-2t_i+t_{i+1}\\ &=\varepsilon_{i-1}-\varepsilon_i.
\end{align*}
Since $1=\Delta_K(1)=\alpha_0+2\sum_{i>0}\alpha_i$, we can also get
\[\alpha_0=1-2\sum_{i>0}\alpha_i.\]
Thus
\begin{eqnarray}\label{eq:Delta-1}
\Delta_K(-1)&=\alpha_0+2\sum_{i>0}(-1)^i\alpha_i \nonumber \\ &=1-4\sum_{i\ge0}(-1)^{i}\varepsilon_i.
\end{eqnarray}

Given a knot $K\subset S^3$ and an integer $n>0$,
there is an affine isomorphism \cite{OSzAbGr}
\[
\phi:\mathbb Z/n\mathbb Z\to \Spin^\C(S^3_n(K)).
\]
For simplicity, let $d(S^3_n(K),i)=d(S^3_n(K),\phi(i))$.

From \cite{OSzAbGr}, we have
\begin{equation}\label{eq:pCorr}
d(L(n,1),i)=-\frac14+\frac{(2i-n)^2}{4n}.
\end{equation}
Using \cite[Theorem~1.2]{OSzRatSurg}, we get
\begin{equation}\label{eq:nSurgCorr}
d(S^3_n(K),i)=d(L(n,1),i)-2t_{\min\{i,n-i\}}.
\end{equation}

\begin{lemma}\label{lem:CorrEqual}
Suppose that $P(p,q)$ is obtained by the $4q$--surgery on $K$. Let $i$ be an integer with $0\le i\le q$. If $i$ is even, we have
\[
d(S^3_{4q}(K),q-i)=d(S^3_{4q}(K),q+i),
\]
and
\[
t_{q-i}-t_{q+i}=\frac i2.
\]
If $i$ is odd, we have
\[
d(S^3_{4q}(K),q-i)=d(S^3_{4q}(K),q+i)\pm1,
\]
and
\[
t_{q-i}-t_{q+i}=\frac{i\mp1}2.
\]
\end{lemma}
\begin{proof}
Since $S^3_{4q}(K)$ is a prism manifold, it contains a Klein Bottle. So the order--$2$ element in $H_1(S^3_{4q}(K))$ is represented by a curve in the Klein Bottle, such that the complement of the curve in the Klein Bottle is an annulus.
By \cite[Theorem~1.1]{NiWu}, for any $j\in \mathbb Z/4q\mathbb Z$, we have
\begin{equation}\label{eq:CorrDiff}
|d(S^3_{4q}(K),j)-d(S^3_{4q}(K),j+2q)|\le1.
\end{equation}
Since the conjugate of $\phi(j+2q)$ is $\phi(2q-j)$, we have
\begin{equation}\label{eq:Conj}
d(S^3_{4q}(K),j+2q)=d(S^3_{4q}(K),2q-j).
\end{equation}
Let $j=q-i$. Using~\eqref{eq:pCorr}~and~\eqref{eq:nSurgCorr}, we get
\begin{eqnarray*}
&&d(S^3_{4q}(K),q-i)-d(S^3_{4q}(K),q+i)\\
&=&-\frac14+\frac{(2q-2i-4q)^2}{16q}-2t_{q-i}-\left(
-\frac14+\frac{(2q+2i-4q)^2}{16q}-2t_{q+i}\right)\\
&=&i-2t_{q-i}+2t_{q+i}\in\mathbb Z.
\end{eqnarray*}
Using (\ref{eq:CorrDiff}) and (\ref{eq:Conj}), we get our conclusion.
\end{proof}


\subsection{The proof of Proposition~\ref{prop:Pbound}}

\begin{proof}[Proof of Proposition~\ref{prop:Pbound}]
By Lemma~\ref{lem:CorrEqual} and (\ref{eq:t2q}),
\[
t_0=t_0-t_{2q}\le\left\lfloor \frac{q+1}2\right\rfloor.
\]
By \cite[Lemma~6.1]{NiZhang}, $p=|\Delta_K(-1)|$. Using (\ref{eq:Delta-1}), we get
\begin{align*}
p&\le1+4\sum_{i\ge0}\varepsilon_i\\ &=1+4t_0\\ &\le1+4\left\lfloor \frac {q+1}2\right\rfloor.
\end{align*}

When $q$ is even, $p\le2q+1$. When $q$ is odd, $p\le2q+3$. By Lemma~\ref{lem:mod4}, $p\ne2q+3$, so we must have $p\le2q+1$.
\end{proof}
\begin{proof}[Proof of Theorem~\ref{thm:Pbound}]
The first statement is Proposition~\ref{prop:Pbound}. The second statement follows from combining~\cite[Theorem~1.6]{NiZhang}~and~\cite[Lemma~2.1]{Prism2016}.
\end{proof}

\section{Input from lattice theory}\label{sec:latticeinput}
This section assembles facts about lattices that will be used in the paper. We mainly follow the treatment of~\cite{GreeneCabling,greene:LSRP, Prism2016, Prism2017}.

Recall that an {\it integral lattice} is a finitely generated free abelian group $L$ endowed with a positive definite symmetric bilinear form $\langle,\rangle:L\times L \to \Z$. Given $v \in L$, let $|v| = \langle v, v \rangle$ be the {\it norm} of $v$.
We can extend $\langle,\rangle$ to a  $\Q$--valued pairing on $L\otimes \Q$; using it we define
\[
L^*=\{x\in L\otimes \Q|\langle x,y\rangle\in \Z, \forall y\in L\}.
\]
The pairing on $L$ descends to a non--degenerate, symmetric bilinear form on the {\it discriminant group} $\overline L=L^*/L$
\begin{align*}
b:\overline L \times \overline L \to \Q/\Z\\
b(\overline x, \overline y)\equiv\langle x,y\rangle \pmod1,
\end{align*}
the {\it linking form}, where $\overline x$ denotes the class of $x\in L$ in $\overline L$. The {\it discriminant} of $L$ is the order of the finite group $\overline L$. Let
\[
\Ch(L)=\{x\in L^*|\langle x,y\rangle\equiv\langle y,y\rangle \pmod2, \forall y\in L\}
\]
denote the set of {\it characteristic covectors} for $L$. The set $C(L)=\Ch(L)/2L$ forms a torsor over the discriminant group $\overline L$. Given $\chi\in C(L)$, define
\begin{equation}\label{eq:d}
d_L([\chi])=\displaystyle \min\left \{\left.\frac{|\chi'|-\text{rk}(L)}{4}\right|\chi'\in [\chi] \right \},
\end{equation}
and call an element $\chi\in \Ch(L)$ {\it short} if its norm is minimal in $[\chi]$. We call the pair $(C(L), d_L)$ the {\it d--invariant} of the lattice $L$; in particular it is an invariant of the stable isomorphism type of the lattice~$L$~\cite[Theorem~4.7]{OSzBrDoub}. We drop $L$ from the notation when the lattice $L$ is understood from the context.

\subsection{Linear lattices}
Given a pair of relatively prime positive integers $p,q$, 
write $\frac{p}{q}$ in a Hirzebruch--Jung continued fraction
\begin{equation}\label{eq:ContFrac}
\displaystyle \frac{p}q = a_{-1} - \frac{1}{a_0 - \displaystyle\frac{1}{\ddots - \displaystyle\frac{1}{a_n}}} = [a_{-1},a_0,\dots,a_n]^-,
\end{equation}
with $a_i\ge 2$ when $i\ge0$ in Equation~\eqref{eq:ContFrac}.

\begin{definition}\label{def:CType}
The {\it linear lattice} $\Lambda(q,-p)$ has a basis
\begin{equation}\label{VertexBasis}
\{x_0,\dots,x_n\},
\end{equation}
and inner product given by
\begin{equation}\label{eq:ai}
\braket{x_i}{x_j} = \begin{cases}
a_i, & i = j \\
-1, & |i-j| = 1\\
0, &|i-j|>1,
\end{cases}
\end{equation}
where the coefficients $a_i$, for $i\in \{0,\cdots,n\}$, are defined by the continued fraction~\eqref{eq:ContFrac}. We call~\eqref{VertexBasis} the \emph{vertex basis} of $\Lambda(q,-p)$.
\end{definition}

\begin{rmk}
The reason that we use $\Lambda(q,-p)$ instead of $\Lambda(q,p)$ is that our convention for lens spaces is different from that of~\cite{greene:LSRP}. In our paper, the lens space $L(q,p)$ is oriented as the $\frac qp$--surgery on the unknot, and $P(p,q)$ is the $\frac qp$--surgery on $\mathbb RP^1\#\mathbb RP^1\subset \mathbb RP^3\#\mathbb RP^3$, so they both bound 4--manifolds with intersection lattice $\Lambda(q,-p)$.
\end{rmk}

An element $\ell \in L$ is {\it reducible} if $\ell = x+y$ for some nonzero $x, y \in L$, with $\langle x, y \rangle \ge 0$, and {\it irreducible} otherwise. An element $\ell \in L$ is {\it breakable} if $\ell = x+y$ with $|x|, |y| \ge 3$ and $\langle x, y \rangle =-1$, and {\it unbreakable} otherwise.
\begin{definition}
In a linear lattice, if $I$ is any subset of $\{x_0,x_1,\dots,x_n\}$ then write $[I] = \sum_{x \in A} x$. An {\it interval} is an element of the form $[I]$ with $I = \{x_a,x_{a+1},\dots,x_b\}$ for $0 \le a \le b \le n$. We say that $a$ is the left endpoint of the interval, and $b$ is the right endpoint of the interval. Say that $[I]$ contains $x_i$ if $I$ does: we often write $x_i\in [I]$ in this case.
\end{definition}

\begin{prop}\label{prop:IntervalsIrreducible}\cite[Proposition~3.3]{greene:LSRP}
If $v \in \Lambda(q,-p)$ is irreducible, $v = \epsilon[I]$ for some $\epsilon = \pm 1$ and $[I]$ an interval.
\end{prop}

From now on, let $[v]$ be the interval corresponding to $v$
when $v$ is irreducible.

\begin{definition}\label{Def:HighNorm}
A vertex $x_i$ has {\it high weight} if $|x_i| = a_i > 2$.
\end{definition}

\begin{prop}\cite[Corollary~3.5(4)]{greene:LSRP}\label{prop:Unbreakable}
An element $\epsilon[I] \in \Lambda(q,-p)$ with $\epsilon \in \{\pm 1 \}$ is unbreakable if and only if $[I]$ contains at most one element of high weight.
\end{prop}

\begin{definition}
For two intervals $[I]$ and $[J]$ with left endpoints $i_0,j_0$ and right endpoints $i_1,j_1$, say that $[I]$ and $[J]$ are {\it distant} if either $i_1 + 1 < j_0$ or $j_1 + 1 < i_0$, that $[I]$ and $[J]$ {\it share a common end} if $i_0 = j_0$ or $i_1 = j_1$, and that $[I]$ and $[J]$ are {\it consecutive} if $i_1 + 1 = j_0$ or $j_1 + 1 = i_0$. Write $[I] \prec [J]$ if $I \subset J$ and $[I]$ and $[J]$ share a common end, and $[I] \dagger [J]$ if they are consecutive. If $[I]$ and $[J]$ are either consecutive or share a common end, say that they {\it abut}. If $I \cap J$ is nonempty and $[I]$ and $[J]$ do not share a common end, write $[I] \pitchfork [J]$.
\end{definition}

\begin{prop}\label{indecomposable}\cite[Corollary~3.5(2)]{greene:LSRP}
The lattice $\Lambda(q,-p)$ is indecomposable; that is, $\Lambda(q,-p)$ is not the direct sum of two nontrivial lattices.
\end{prop}

\begin{prop}[Proposition~3.6~of~\cite{greene:LSRP}]\label{pp}
If $\Lambda (q, p) \cong \Lambda(q', p')$, then $q = q'$ and either $p \equiv p'$ or $pp' \equiv 1\text{ (mod $q$)}$.
\end{prop}

\subsection{Changemaker lattices}
When a lattice $L$ is isomorphic to $\sigma^\perp$, the orthogonal complement of a changemaker vector $\sigma\in \Z^{n+2}$, $L$ is called a {\it changemaker lattice}.
\begin{definition}\label{stbasis}
The {\it standard basis} of $\sigma^\perp$ is the collection $S = \{v_1, \dots, v_{n+1}\}$, where
\begin{equation*}
    v_j = \left(2e_0 + \sum_{i = 1}^{j - 1} e_i\right) - e_j
\end{equation*}
whenever $\sigma_j = 1 + \sigma_0 + \cdots + \sigma_{j-1}$, and
\begin{equation*}
    v_j = \left(\sum_{i \in A} e_i\right) - e_j
\end{equation*}
whenever $\sigma_j = \sum_{i \in A} \sigma_i$, with $A \subset \{0, \dots, j-1\}$ chosen to maximize the quantity $\sum_{i \in A} 2^i$.  A vector $v_j \in S$ is called {\it tight} in the first case, {\it just right} in the second case as long as $i < j-1$ and $i \in A$ implies that $i+1\in A$, and {\it gappy} if there is some index $i$ with $i \in A$, $i < j-1$, and $i+1 \not \in A$. Such an index, $i$, is a {\it gappy~index} for $v_j$.
\end{definition}

\begin{definition}
For $v \in \Z^{n+2}$, $\supp v = \{i | \braket{e_i}{v} \neq 0\}$, $\supp^+ v = \{i | \braket{e_i}{v} > 0\}$, and $\supp^- v = \{i | \braket{e_i}{v} > 0\}$.
\end{definition}

\begin{lemma}[Lemma~3.12~(3) in \cite{greene:LSRP}] \label{gappy3}
If $|v_{k+1}|=2$, then $k$ is not a gappy index for any $v_j$ with $j \in \{1, \cdots, n+1 \}$.
\end{lemma}

\begin{lemma}[Lemma~3.13 in \cite{greene:LSRP}] \label{lem:irred}
Each $v_j \in S$ is irreducible. In fact, suppose $A\subset\{0,1,\dots,j-1\}$, then the vector
\[-e_j+\sum_{i\in A}e_i\]
is irreducible.
\end{lemma}

\begin{lemma}\label{lem:Decomp}
Let $v=\sum_{i\in A} b_ie_i\in L$, with $A\subset\{0,1,\cdots,n+1\}$ and each $b_i\in\{-1,1\}$.
If $v=x+y$ with $\braket xy\ge0$, then there exists a subset $B\subset A$ such that
\[x=\sum_{i\in B} b_ie_i, y=\sum_{i\in A\setminus B} b_ie_i.\]
\end{lemma}
\begin{proof}
Let $x=\sum x_ie_i,y=\sum y_ie_i$. Since $x_i+y_i\in\{-1,0,1\}$, $x_iy_i\le0$. If $\braket xy\ge0$, then each $x_iy_i=0$, namely, one of $x_i,y_i$ is $0$. So our conclusion holds.
\end{proof}

\begin{lemma}[Lemma~3.15 in \cite{greene:LSRP}]\label{lem:BrIsTight}
If $v_j \in S$ is breakable, then it is tight.
\end{lemma}

\begin{lemma}[Lemma~4.2(1) in \cite{greene:LSRP}]\label{tight1}
If $\Lambda(q,-p)$ is a changemaker lattice, then it contains at most one tight vector.
\end{lemma}

\begin{lemma}[Lemma~3.12(1) in \cite{greene:LSRP}]\label{lem:j-1}
For any $v_j\in S$, we have $j-1\in\supp(v_j)$.
\end{lemma}

\begin{definition}
If $T$ is a set of irreducible vectors in a linear lattice $\Lambda(q,-p)$, the {\it intersection graph} $G(T)$ has vertex set $T$, and an edge between $v$ and $w$ if the intervals corresponding to $v$ and $w$ abut. We write $v\sim w$ if $v$ and $w$ are connected in $G(T)$.
\end{definition}

\begin{lemma}\label{lem:PairingNonzero}
If the intervals corresponding to $v$ and $w$ abut, then $\braket vw\ne0$.
\end{lemma}

\begin{lemma}[Lemma~4.4 in \cite{greene:LSRP}]\label{lem:Pairing1}
If $v_i$ and $v_j$ are distinct unbreakable vectors with $|v_i|,|v_j|\ge3$, then $|\braket{v_i}{v_j}|\le1$, with equality if and only if $[v_i]\dagger[v_j]$.
\end{lemma}

\begin{lemma}[Corollary~4.5 in \cite{greene:LSRP}]\label{lem:DistinctHighNorm}
If $v_i$ and $v_j$ are distinct unbreakable vectors with $|v_i|,|v_j|\ge3$, then the high weight vertices contained in $v_i,v_j$ are different.
\end{lemma}

\begin{definition}
A {\it claw} in a graph $G$ is a quadruple $(v;w_1,w_2,w_3)$ of vertices such that $v$ neighbors all the $w_i$, but no two of the $w_i$ neighbor each other.
\end{definition}

\begin{lemma}[Lemma~4.8 of~\cite{greene:LSRP}]\label{claw}
The intersection graph $G(T)$ has no claws.
\end{lemma}

\begin{definition}
Given a set $T$ of unbreakable elements in a linear lattice and $v_1,v_2,v_3 \in T$, $(v_1,v_2,v_3)$ is a {\it heavy triple} if $|v_i| \ge 3$, and if each pair among the $v_i$ is connected by a path in $G(T)$ disjoint from the third.
\end{definition}

\begin{lemma}[Based on Lemma~4.10 of~\cite{greene:LSRP}]\label{heavytriple}
$G(T)$ has no heavy triples.
\end{lemma}

\section{The topology of certain cobordisms}\label{sec:cobordism}

In this section, we will consider the topology of a certain cobordism
$W: Y_0\to Y_1$. 
We assume that $W$ is obtained by adding $n+1$ two-handles along a link $L\subset Y_0$, such that one component $L_0$ of $L$ represents a $2$--torsion in $H_1(Y_0)$, and all other components of $L$ are null-homologous in $Y_0$. Moreover, we assume that $|H_1(Y_0)|=4$ and $W$ is negative definite. So $Y_1$ is a rational homology sphere. Let $\iota_i: Y_i\to W$ be the inclusion map, $\iota_i^*: H^2(W)\to H^2(Y_i)$ be the induced maps on cohomology, and $\iota_i^s:\Spin^\C(W)\to\Spin^\C(Y_i)$ be the induced maps on $\Spin^\C$, $i=0,1$.

We make the further assumption that $Y_0$ is the boundary of a compact $4$--manifold $Z_0$ with $H_1(Z_0)\cong\mathbb Z/2\mathbb Z$ and $H_2(Z_0)=0$, and $L_0$ is null-homologous in $Z_0$. Let $Z=Z_0\cup_{Y_0}W$.

From the handle structure of $W$, we can compute
\[
H_1(W)\cong\mathbb Z/2\mathbb Z, H_2(W)\cong\mathbb Z^{n+1}, H_1(W,Y_i)=0, H_2(W,Y_i)\cong\mathbb Z^{n+1}, i=0,1 .
\]
By the Universal Coefficient Theorem,
\[H^2(W)\cong \mathbb Z^{n+1}\oplus\mathbb Z/2\mathbb Z.\]
In particular, there exists a unique torsion class $\alpha\in H^2(W)$. Let $\alpha_i=\iota_i^*(\alpha)$, $i=0,1$.

Since $Z$ is obtained by adding two-handles to $Z_0$, such that all attaching curves are null-homologous in $Z_0$, we have 
\[H_1(Z)\cong H_1(Z_0)\cong\mathbb Z/2\mathbb Z,\] and the map $H_2(Z)\to H_2(Z,Z_0)$ is an isomorphism. 

\begin{lemma}\label{lem:ZtoW}
The map $\iota^*_{W,Z}: H^2(Z)\to H^2(W)$ is injective with image containing $\alpha$. The map $\iota_{Y_0,Z_0}^*: H^2(Z_0)\to H^2(Y_0)$ is injective with image generated by $\alpha_0$. Moreover, $[L_0]\in H_1(Y_0)$ is the Poincar\'e dual of $\alpha_0$.
\end{lemma}
\begin{proof}
Using the long exact sequences
\[
H^2(Z,W)\to H^2(Z)\to H^2(W),\quad H^2(Z_0,Y_0)\to H^2(Z_0)\to H^2(Y_0),
\]
and the fact that $0=H^2(Z_0,Y_0)\cong H^2(Z,W)$, we get that $\iota^*_{W,Z}$ and $\iota_{Y_0,Z_0}^*$ are injective.

By the Universal Coefficient Theorem,  $H^2(Z)\cong Hom(H_2(Z),\mathbb Z)\oplus \mathbb Z/2\mathbb Z$, so it has a unique $2$--torsion $\overline{\alpha}$. Since $\iota^*_{W,Z}$ is injective, $\iota^*_{W,Z}(\overline{\alpha})$ is a $2$--torsion in $H^2(W)$, which must be $\alpha$. Let $\overline{\alpha}_0$ be the restriction of $\overline{\alpha}$ to $H^2(Z_0)$. Using the commutative diagram
\[
\xymatrix{
H^2(Z)\ar[r]\ar[d] &H^2(Z_0)\ar[d]\\
H^2(W)\ar[r] &H^2(Y_0)
},
\]
we see that $\iota_{Y_0,Z_0}^*(\overline{\alpha}_0)=\alpha_0$. 
Since $H^2(Z_0)\cong\Z/2\Z$, the image of $\iota_{Y_0,Z_0}^*$ is generated by $\alpha_0$.

Since $L_0$ is null-homologous in $Z_0$, there exists a properly embedded oriented surface $F_0\subset Z_0$ such that $\partial F_0=L_0$. Thus the image of the Poincar\'e dual of $[F_0]$ under $\iota_{Y_0,Z_0}^*$ is the Poincar\'e dual of $[L_0]$. Since both $[L_0]$ and $[\alpha_0]$ have order $2$, and $\iota_{Y_0,Z_0}^*(\overline{\alpha}_0)=\alpha_0$, we get that $[L_0]$ is the Poincar\'e dual of $\alpha_0$.
\end{proof}

\begin{lemma}\label{lem:KerIota}
(1) For $i=0,1$, we have $\ker\iota_i^*\cong H^2(W,Y_i)$, and $\iota_i^*$ is surjective. In particular, $\alpha_i\ne0$ in $H^2(Y_i)$.
\newline(2) The kernel of the restriction map $(\iota_0')^*:\ker\iota_1^*\to H^2(Y_0)$ is isomorphic to $H^2(W,\partial W)$, and its image is generated by $\alpha_0$.
\end{lemma}
\begin{proof}
(1) The first statement follows from the long exact sequence
\[
0=H^1(Y_i)\to H^2(W,Y_i)\to H^2(W)\stackrel{\iota_i^*}{\longrightarrow} H^2(Y_i)\to H^3(W,Y_i)=0.
\]
It follows that $\ker\iota_i^*$ is torision-free, so $\alpha\notin\ker\iota_i^*$. Thus $\alpha_i\ne0$.

(2) By (1), the map $(\iota_0')^*$ can be identified with $H^2(W,Y_1)\to H^2(Y_0)$, which is part of the long exact sequence
\begin{equation*}
0=H^1(\partial W,Y_1) \to
H^2(W,\partial W) \to H^2(W,Y_1) \to H^2(\partial W,Y_1)=H^2(Y_0).
\end{equation*}
Thus $\ker(\iota_0')^*$ is $H^2(W,\partial W)$. 

By Poincar\'e duality, $(\iota_0')^*$ can be identified with the boundary map $\partial_0': H_2(W,Y_0)\to H_1(Y_0)$. By the handle decomposition of $W$, we see that the image of $\partial_0'$ is generated by $[L_0]$.
By Lemma~\ref{lem:ZtoW}, $\mathrm{im}(\iota_0')^*$ is generated by $\alpha_0$.
\end{proof}

\begin{cor}\label{cor:SpincCorb}
For each $\mathfrak t\in\Spin^\C(Y_1)$, there exists a subset \[\mathfrak R(\mathfrak t)=\{\mathfrak r_0,\mathfrak r_1=\mathfrak r_0+\alpha_0\}\subset\Spin^\C(Y_0)\] such that for each $\mathfrak r\in\Spin^\C(Y_0)$, the set
\begin{equation}\label{eq:SpincInter}
(\iota_0^s,\iota_1^s)^{-1}(\mathfrak r,\mathfrak t):=(\iota_0^s)^{-1}(\mathfrak r)\cap (\iota_1^s)^{-1}(\mathfrak t)
\end{equation}
is nonempty if and only if $\mathfrak r\in\mathfrak R(\mathfrak t)$. Moreover, the set (\ref{eq:SpincInter}) is an $H^2(W,\partial W)$--torsor when it is nonempty.
\end{cor}
\begin{proof}
This follows from Lemma~\ref{lem:KerIota} and the fact that $\Spin^\C$ is an $H^2$--torsor.
\end{proof}

By the long exact sequence
\[
0=H_2(Y_0)\to H_2(W)\to H_2(W,Y_0)\to H_1(Y_0),
\]
$H_2(W)$ embeds as an index--$2$ subgroup of $H_2(W,Y_0)\cong\mathbb Z^{n+1}$. Thus we can extend the intersection form on $H_2(W)$ to  $H_2(W,Y_0)$, with value in $\frac14\mathbb Z$.
Let 
\[\mathcal L\cong H_2(W,Y_0)\cong H_2(Z,Z_0)\cong H_2(Z)\] 
be the intersection lattice on the pair $(W,Y_0)$. Suppose that the generators corresponding to the two-handles are $x_0,\dots,x_n$, where $x_0$ corresponds to the two-handle attached along $L_0$. Let \[\mathcal L_0=\langle2x_0,x_1,\dots,x_n\rangle\] be the sublattice of $\mathcal L$ generated by $2x_0,x_1,\dots,x_n$; then $\mathcal L_0$ can be identified with the intersection lattice $H_2(W)$. Let \[\mathcal L^*=Hom(\mathcal L,\mathbb Z),\mathcal L^*_0=Hom(\mathcal L_0,\mathbb Z)\supset \mathcal L^*.\] Using the inner product on $\mathcal L$, we can embed $\mathcal L^*$ and $\mathcal L_0^*$ as sublattices of $\mathcal L\otimes\mathbb Q$.

Let
\[\widetilde{\mathcal C}=\{y\in \mathcal L_0^*|\braket{y}{2x_0}\equiv\braket{2x_0}{2x_0}, \braket{y}{x_j}\equiv\braket{x_j}{x_j}\pmod2,\quad j>0\}.\]
Let $\overline H^2(W)=H^2(W)/\mathrm{Tors}=\mathcal L_0^*$, and let $\bar c_1: \Spin^\C(W)\to \overline H^2(W)$ be the composition of the map $c_1:\Spin^\C(W)\to H^2(W)$ and the quotient map $H^2(W)\to \overline H^2(W)$.
Then $\widetilde{\mathcal C}$ is the image of $\bar c_1$.

\begin{prop}\label{prop:SpincIden}
(1) The quotient $\Spin^\C(Y_1)/\langle\alpha_1\rangle$ can be identified with $\widetilde{\mathcal C}/2\mathcal L$.
\newline(2) Under the previous identification, suppose that the $\langle\alpha_1\rangle$--orbit $\{\mathfrak t,\mathfrak t+\alpha_1\}$ is identified with $y+2\mathcal L$ for some $y\in\widetilde{\mathcal C}$. Let $\mathfrak R(\mathfrak t)=\{\mathfrak r_0,\mathfrak r_1\}$. Then there exist $y_0,y_1\in y+2\mathcal L$, such that
\[
\bar c_1((\iota_0^s,\iota_1^s)^{-1}(\mathfrak r_0,\mathfrak t))=y_0+2\mathcal L_0, \quad\bar c_1((\iota_0^s,\iota_1^s)^{-1}(\mathfrak r_1,\mathfrak t))=y_1+2\mathcal L_0,
\]
and
\[
\bar c_1((\iota_0^s,\iota_1^s)^{-1}(\mathfrak r_0,\mathfrak t+\alpha_1))=y_1+2\mathcal L_0, \quad\bar c_1((\iota_0^s,\iota_1^s)^{-1}(\mathfrak r_1,\mathfrak t+\alpha_1))=y_0+2\mathcal L_0.
\]
\end{prop}
\begin{proof}
(1) By Lemma~\ref{lem:KerIota}, every $\mathfrak t\in\Spin^\C(Y_1)$ is in the image of $\iota^s_1$, and $\mathfrak s_1,\mathfrak s_2\in\Spin^\C(W)$ restrict to the same $\mathfrak t\in\Spin^\C(Y_1)$ if and only if $\mathfrak s_1-\mathfrak s_2\in H^2(W,Y_1)\cong H_2(W,Y_0)=\mathcal L$. So $\Spin^\C(Y_1)\cong \Spin^\C(W)/\mathcal L$. Consider the map $\bar c_1: \Spin^\C(W)\to \widetilde{\mathcal C}$. It is surjective, and $\bar c_1(\mathfrak s_1)=\bar c_1(\mathfrak s_2)$ if and only if  $\mathfrak s_1-\mathfrak s_2\in \langle\alpha\rangle$. Using the formula
\[
c_1(\mathfrak s_1)-c_1(\mathfrak s_2)=2(\mathfrak s_1-\mathfrak s_2)
\]
we get that $\Spin^\C(Y_1)/\langle\alpha_1\rangle\cong\Spin^{\C}(W)/(\mathcal L+\langle\alpha\rangle)\cong\widetilde{\mathcal C}/2\mathcal L$.

(2) By Corollary~\ref{cor:SpincCorb}, there exist $\mathfrak s_0,\mathfrak s_1\in\Spin^\C(W)$, such that
\[
(\iota_0^s,\iota_1^s)^{-1}(\mathfrak r_0,\mathfrak t)=\mathfrak s_0+\mathcal L_0,\quad(\iota_0^s,\iota_1^s)^{-1}(\mathfrak r_1,\mathfrak t)=\mathfrak s_1+\mathcal L_0.
\]
Since
\[\iota_0^s(\mathfrak s_1+\alpha)=\iota_0^s(\mathfrak s_1)+\alpha_0=\mathfrak r_1+\alpha_0=\mathfrak r_0, \iota_0^s(\mathfrak s_0+\alpha)=\mathfrak r_1,\]
we also have
\[
(\iota_0^s,\iota_1^s)^{-1}(\mathfrak r_0,\mathfrak t+\alpha_1)=\mathfrak s_1+\alpha+\mathcal L_0,\quad(\iota_0^s,\iota_1^s)^{-1}(\mathfrak r_1,\mathfrak t+\alpha_1)=\mathfrak s_0+\alpha+\mathcal L_0.
\]
Applying $\bar c_1$ to the above equalities, we get our conclusion.
\end{proof}

For any $\mathfrak s\in\Spin^\C(W)$, let
\begin{equation}\label{eq:gr}
gr(W,\mathfrak s)=\frac{c_1^2(\mathfrak s)+b_2(W)}4.
\end{equation}
For any $\mathfrak t\in\Spin^\C(Y_1)$, let
\begin{equation}\label{eq:D}
D_W(Y_1,\mathfrak t)=\max_{
\begin{array}{c}
\scriptstyle\mathfrak s\in\Spin^\C(W)\\
\scriptstyle\mathfrak s|Y_1=\mathfrak t
\end{array}
} (d(Y_0,\mathfrak s|Y_0)+gr(W,\mathfrak s)).
\end{equation}

\begin{lemma}\label{lem:RatBall}
There are exactly two Spin$^c$ structures $\mathfrak e_0,\mathfrak e_1\in\Spin^\C(Y_0)$ which can be extended over $Z_0$. Moreover,
\[\mathfrak e_1=\mathfrak e_0+\alpha_0, \quad d(Y_0,\mathfrak e_i)=0,\quad i=0,1.\]
\end{lemma}
\begin{proof}
By Lemma~\ref{lem:ZtoW}, $\alpha_0$ is the restriction of a cohomology class in $H^2(Z_0)$.
Let $\mathfrak e_0\in\Spin^\C(Y_0)$ be a Spin$^c$ structure which is the restriction of a Spin$^c$ structure on $Z_0$, then $\mathfrak e_1:=\mathfrak e_0+\alpha_0$ also extends over $Z_0$. Since $H^2(Z_0)\cong\Z/2\Z$, $\mathfrak e_0,\mathfrak e_1$ are the only two Spin$^c$ structures which can be extended over $Z_0$.
It follows from \cite[Proposition~9.9]{OSzAbGr} that $d(Y_0,\mathfrak e_i)=0$.
\end{proof}

\begin{lemma}\label{lem:ExtendZ0}
The image of \[\bar c_1: (\iota^s_0)^{-1}(\{\mathfrak e_0,\mathfrak e_1\})\to \overline H^2(W)\] is $\mathcal C:=\Ch(\mathcal L)$.
\end{lemma}
\begin{proof}
Let $\mathfrak s_0$ be the restriction of a Spin$^c$ structure on $Z$ to $W$, then $\mathfrak s_0\in(\iota^s_0)^{-1}(\{\mathfrak e_0,\mathfrak e_1\})$. Clearly, $\bar c_1(\mathfrak s_0)\in\mathcal C$. By Lemma~\ref{lem:ZtoW}, $\iota^*_{W,Z}$ is injective, so the image of $H^2(Z)$ in $\overline H^2(W)$ can be identified with $Hom(H_2(Z),\Z)=Hom(H_2(W,Y_0),\Z)=\mathcal L^*$. Thus
$\bar c_1((\iota^s_0)^{-1}(\{\mathfrak e_0,\mathfrak e_1\}))$ is a $2\mathcal L^*$--torsor. Since $\mathcal C$ is the unique $2\mathcal L^*$--torsor containing $\bar c_1(\mathfrak s_0)$, our conclusion holds.
\end{proof}

\begin{cor}\label{cor:SumInv}
The sum
\begin{equation}\label{eq:DSum}
\sum_{\mathfrak t\in\Spin^\C(Y_1)}D_W(Y_1,\mathfrak t)
\end{equation}
only depends on the lattice $\mathcal L$ and the correction terms of $Y_0$.
In fact, if we write (\ref{eq:DSum}) as a function
\[
\mathscr D(\mathcal L,\{d_0,d_1\})
\]
of $\mathcal L$ and the multiset $\{d_0,d_1\}$ of the  correction terms of the two Spin$^c$ structures other than $\mathfrak e_0,\mathfrak e_1$, then
\begin{equation}\label{eq:cShift}
\mathscr D(\mathcal L,\{d_0+c,d_1+c\})=\mathscr D(\mathcal L,\{d_0,d_1\})+c|\mathcal L_0^*/\mathcal L|
\end{equation}
for any $c\in\mathbb Q$. Note that, by Proposition~\ref{prop:SpincIden}, $|H_1(Y_1)|=2|\mathcal L_0^*/\mathcal L|$.
\end{cor}
\begin{proof}
We will give the procedure of computing (\ref{eq:DSum}) from $\mathcal L$ and the correction terms of $Y_0$. Let $\mathfrak o_0,\mathfrak o_1$ be the two Spin$^c$ structures other than $\mathfrak e_0,\mathfrak e_1$ on $Y_0$.
We choose $[z]\in\widetilde{\mathcal C}/2\mathcal L$. By Proposition~\ref{prop:SpincIden}, $[z]$ corresponds to a pair of Spin$^c$ structures $\mathfrak t_0,\mathfrak t_1=\mathfrak t_0+\alpha_1\in\Spin^\C(Y_1)$. There are exactly two $2\mathcal L_0$--torsors contained in $z+2\mathcal L$, denoted by $\mathcal T_0,\mathcal T_1$.

Next we check whether $z+2\mathcal L$ is contained in $\mathcal C$. If it is contained in $\mathcal C$, it follows from Lemma~\ref{lem:ExtendZ0} that each $\mathfrak t_i$ is cobordant to $\mathfrak e_0$ and $\mathfrak e_1$, $i=0,1$. Since $d(Y_0,\mathfrak e_0)=d(Y_0,\mathfrak e_1)=0$, by Proposition~\ref{prop:SpincIden},
\[
D_W(Y_1,\mathfrak t_0)=D_W(Y_1,\mathfrak t_1)
=0+\max_{y\in z+2\mathcal L}\frac{-\braket{y}{y}+b_2(W)}4.
\]
If $z+2\mathcal L$ is not contained in $\mathcal C$, then each $\mathfrak t_i$ is cobordant to $\mathfrak o_0$ and $\mathfrak o_1$. By Proposition~\ref{prop:SpincIden}, the multiset $\{D_W(Y_1,\mathfrak t_0),D_W(Y_1,\mathfrak t_1)\}$ is equal to
\[
\left\{\max\{d(Y_0,\mathfrak o_0)+\max_{y\in\mathcal T_0}\frac{-\braket{y}{y}+b_2(W)}4, d(Y_0,\mathfrak o_1)+\max_{y\in\mathcal T_1}\frac{-\braket{y}{y}+b_2(W)}4
\},\right.\]
\[\left.
\quad\max\{d(Y_0,\mathfrak o_0)+\max_{y\in\mathcal T_1}\frac{-\braket{y}{y}+b_2(W)}4, d(Y_0,\mathfrak o_1)+\max_{y\in\mathcal T_0}\frac{-\braket{y}{y}+b_2(W)}4
\}
\right\}.
\]

Finally, to get
(\ref{eq:DSum}), we add all the $D_W(Y_1,\mathfrak t_0)+D_W(Y_1,\mathfrak t_1)$ together,  for all $[z]\in\widetilde{\mathcal C}/2\mathcal L$.

The equality (\ref{eq:cShift}) follows from the above procedure, since exactly $\frac12|H_1(Y_1)|$ values of $D_W(Y_1,\mathfrak t)$ are increased by $c$ after inceasing $d(Y_0,\mathfrak o_i)$ by $c$, $i=0,1$.
\end{proof}


\section{Sharp cobordisms}\label{sect:SharpCob}

In this section, we will generalize the notion of sharp $4$--manifolds defined by Greene \cite{GreeneCabling} to $4$--dimensional cobordisms, and prove that certain cobordisms between prism manifolds are sharp. Recall that a smooth, compact, negative definite $4$--manifold $X$ with $\partial X=Y$ is {\it sharp} if for every $\mathfrak t \in \text{Spin}^c(Y)$, there exists some $\mathfrak s\in \text{Spin}^c(X)$ extending $\mathfrak t$ such that
\[
c_1(\mathfrak s)^2 + b_2(X)= 4d(Y, \mathfrak t)
\]
\begin{definition}
Let $W: Y_0\to Y_1$ be a smooth, connected, negative definite cobordism between two rational homology spheres $Y_0$ and $Y_1$. We say $W$ is {\it sharp}, if for any $\mathfrak t\in\Spin^\C(Y_1)$ we have
\[
d(Y_1,\mathfrak t)=D_W(Y_1,\mathfrak t).
\]
Here $D_W$ is defined using the formula (\ref{eq:D}).
\end{definition}

\begin{lemma}\label{lem:DecompSharp}
Let $Y_1,Y_2,Y_3$ be rational homology spheres, $W_1: Y_1\to Y_2$ and $W_2: Y_2\to Y_3$ be two negative definite cobordisms. If $W=W_1\cup_{Y_2}W_2$ is sharp, then $W_2$ is sharp.
\end{lemma}
\begin{proof}
Let $\mathfrak s\in\Spin^\C(W)$ and let $\mathfrak s_i=\mathfrak s|W_i$, $i=1,2$, then
\[
c_1^2(\mathfrak s)=c_1^2(\mathfrak s_1)+c_1^2(\mathfrak s_2).
\]
Our conclusion follows from the the above equality.
\end{proof}

\subsection{A Kirby diagram of $P(p,q)$}\label{subsect:Kirby}

Suppose that \[
\frac{p}{q}=[a_{-1},a_0,\dots,a_n]^-
\] as in (\ref{eq:ContFrac}),
where each $a_i$ is $\ge2$ when $i\ge 0$.

\begin{figure}[ht]
\begin{picture}(340,100)
\put(10,0){\scalebox{0.55}{\includegraphics*
{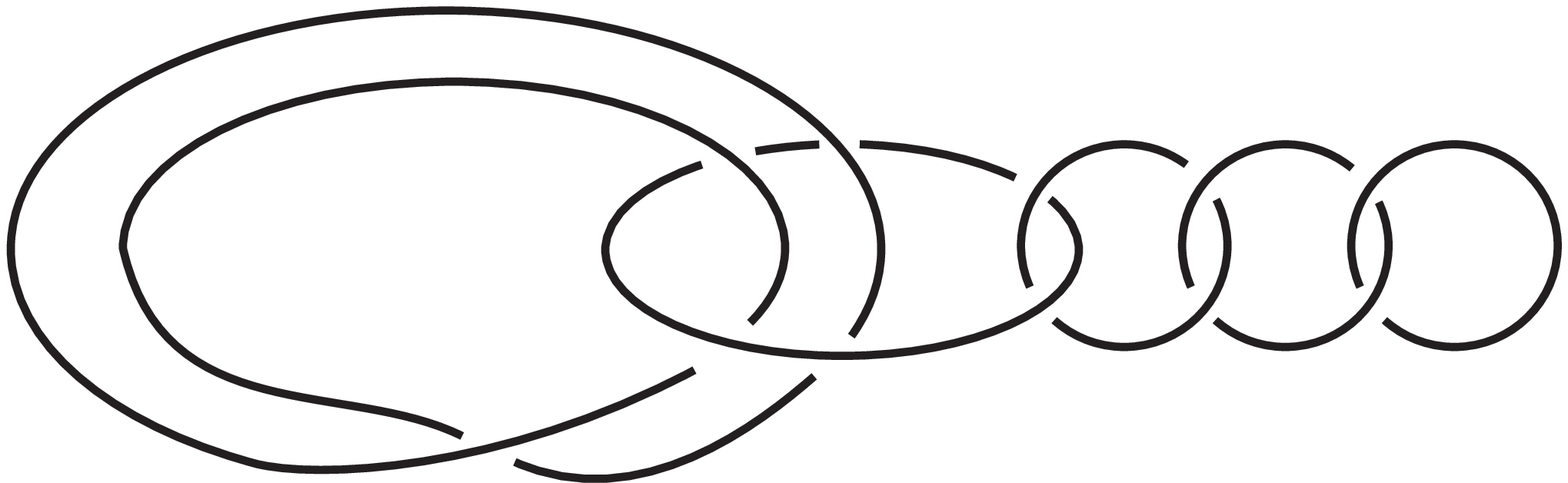}}}

\put(10,20){$0$}

\put(183,15){$-(a_{-1}+2)$}

\put(240,15){$-a_0$}

\put(275,15){$\cdots$}

\put(310,15){$-a_n$}

\end{picture}
\caption{\label{fig:Kirby3}A manifold bounded by $P(p,q)$. If we replace the leftmost component with a dotted circle, we get a negative definite $4$--manifold $Z(p,q)$.}
\end{figure}

Figure~\ref{fig:Kirby3} is a surgery diagram of $P(p,q)$. The leftmost two components give rise to a surgery diagram of $P(a_{-1},1)$, and other components give rise to a negative definite cobordism
\[
W(p,q): P(a_{-1},1)\to P(p,q).
\]

If we replace the leftmost component, which is unknotted with slope $0$, with a dotted circle representing a one-handle, we get a negative definite $4$--manifold $Z(p,q)$ bounded by $P(p,q)$, and the two leftmost components give rise to a rational homology ball $Z_{a_{-1}}$ bounded by $P(a_{-1},1)$, with $H_1(Z_{a_{-1}})=\mathbb Z/2\mathbb Z$.

The main result of this section is the following proposition.

\begin{prop}\label{prop:SharpCobordism}
The cobordism $W(p,q)$ is sharp.
\end{prop}

For simplicity, we only prove the case $q<p<2q$. The proof of the general case is similar.
From now on, let $W=W(p,q)$.

\subsection{More Kirby diagrams}

We will consider $3$ other cobordisms.

When $q<p<2q$, $a_{-1}=2$. We have
\[
\frac{2q-(p-q)}{q-(p-q)}=1+\frac{q}{2q-p}=[a_0+1,a_1,\dots,a_n]^-,
\]
Consider the following surgery diagram of $P(p-q,q)$. By \cite{Prism2017}, this diagram gives rise to a sharp $4$--manifold bounded by $P(p-q,q)$.
 The component with label $-4$ gives rise to $P(1,1)=L(4,-1)$, and the other two-handles give rise to a cobordism
\[W_1: P(1,1)\to P(p-q,q).\]

\begin{figure}[ht]
\begin{picture}(340,100)
\put(10,0){\scalebox{0.55}{\includegraphics*
{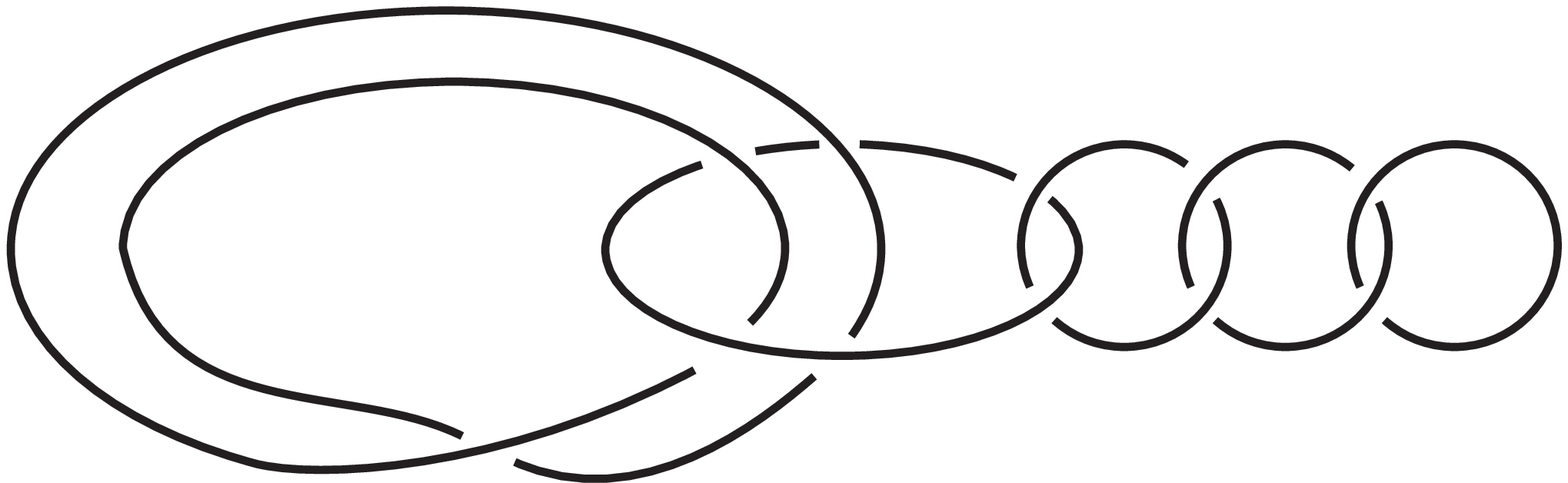}}}

\put(5,20){$-4$}

\put(183,15){$-(a_0+1)$}

\put(240,15){$-a_1$}

\put(275,15){$\cdots$}

\put(310,15){$-a_n$}

\end{picture}
\caption{\label{fig:Kirby4}A sharp $4$--manifold $X(p-q,q)$ bounded by $P(p-q,q)$.}
\end{figure}

\begin{figure}[t]
\begin{picture}(340,170)
\put(90,0){\scalebox{0.7}{\includegraphics*
{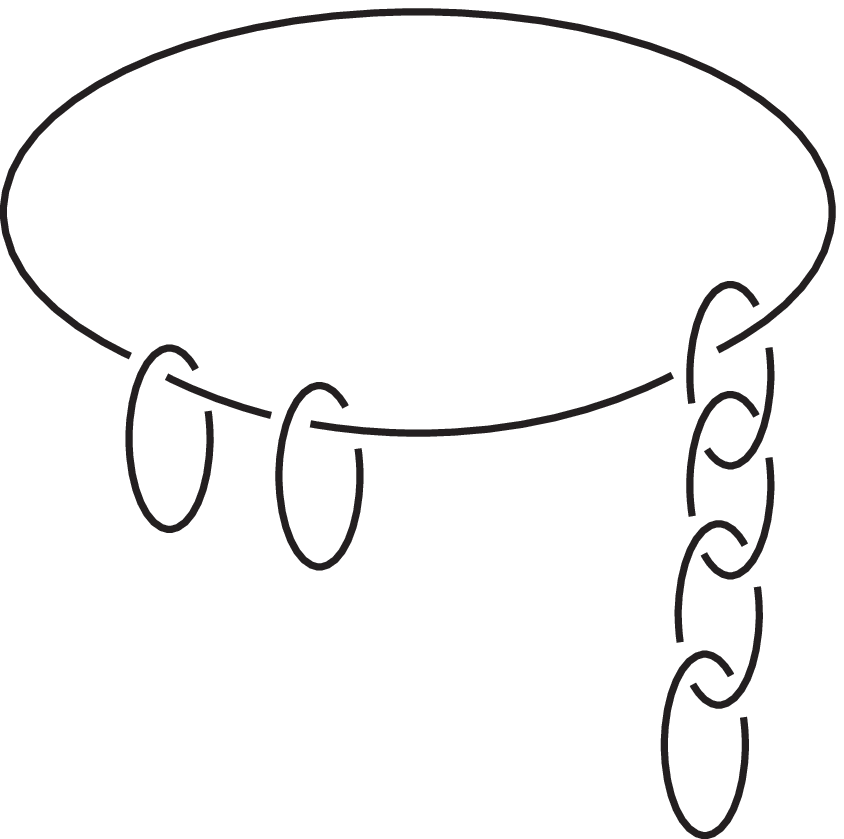}}}

\put(64,130){$-a_0'$}

\put(98,62){$-2$}

\put(150,45){$-2$}

\put(249,90){$-a_1'$}

\put(248,64){$-a_2'$}

\put(254,40){$\vdots$}

\put(244,8){$-a_m'$}
\end{picture}
\caption{\label{fig:Kirby5}A sharp $4$--manifold bounded by $P(p,-q)$}
\end{figure}

\begin{figure}[ht]
\begin{picture}(340,100)
\put(10,0){\scalebox{0.55}{\includegraphics*
{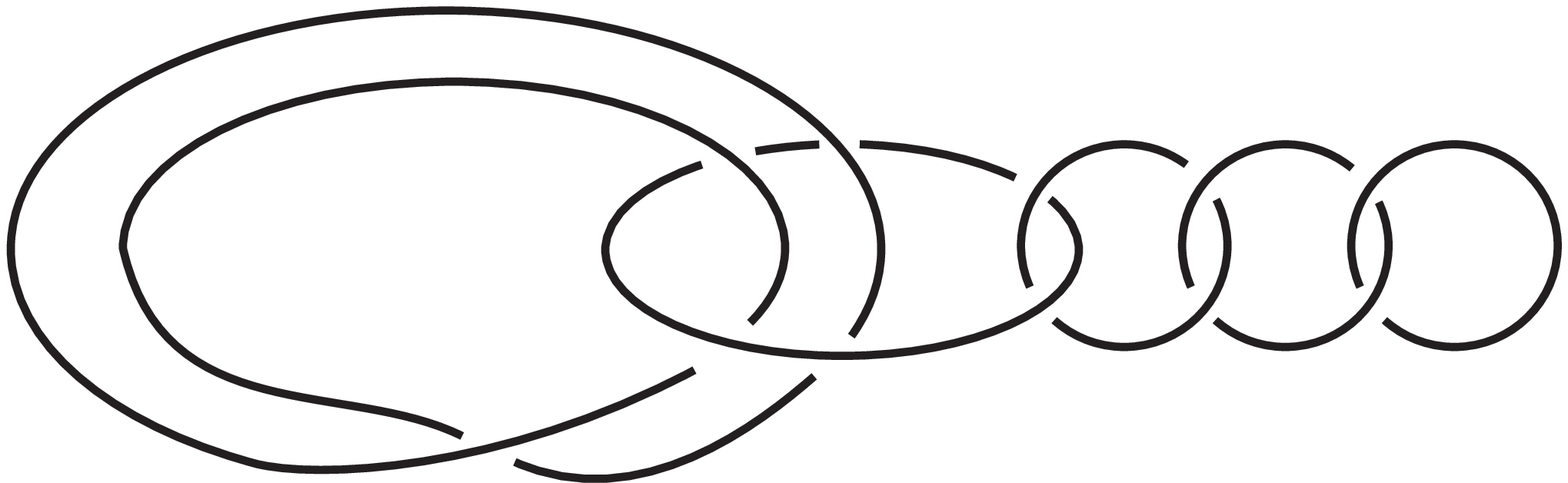}}}

\put(5,20){$-4$}

\put(183,15){$-a'_0$}

\put(240,15){$-a_1'$}

\put(275,15){$\cdots$}

\put(310,15){$-a_m'$}

\end{picture}
\caption{\label{fig:Kirby6}A sharp $4$--manifold bounded by $P(p-q,-q)$.}
\end{figure}

Let
\[
\frac{p+q}{p}=[a_0',a_1',\dots,a_m']^-.
\]
By \cite{Prism2016}, $P(p,-q)$ has a surgery diagram as in Figure~\ref{fig:Kirby5}, which gives rise to a sharp $4$--manifold bounded by $P(p,-q)$.
The two components with label $-2$ give rise to $P(0,1)=\mathbb RP^3\#\mathbb RP^3$, and the other two-handles give rise to a
cobordism
\[W': P(0,1)\to P(p,-q).\]

Using the continued fraction
\[
\frac{-2q-(p-q)}{-q-(p-q)}=\frac{p+q}{p}=[a_0',a_1',\dots,a_m']^-,
\]
by \cite{Prism2017},
we get a surgery diagram of $P(p-q,-q)$ as in Figure~\ref{fig:Kirby6}, which gives rise to a sharp $4$--manifold bounded by $P(p-q,-q)$. The component with label $-4$ gives rise to $P(1,1)=L(4,-1)$, and the other two-handles give rise to a
cobordism
\[W_1':P(1,1)\to P(p-q,-q).\]

By Lemma~\ref{lem:DecompSharp}, $W_1,W',W_1'$ are all sharp cobordisms.

\begin{lemma}\label{lem:IsomH2}
The intersection lattices on $(W,P(2,1))$ and $(W_1,P(1,1))$ are isomorphic; also, the intersection lattices on $(W',P(0,1))$ and $(W_1',P(1,1))$ are isomorphic.
\end{lemma}
\begin{proof}
In Figure~\ref{fig:Kirby3}, consider the knot $L_0$ with label $-a_0$. The canonical longitude on $L_0$ is clearly rationally null-homologous in $P(2,1)\setminus L_0$. As a result, the square of the generator of $H_2(W,P(2,1))$ corresponding to the two-handle attached along $L_0$ is $-a_0$.
In Figure~\ref{fig:Kirby4}, consider the knot $K_0$ with label $-(a_0+1)$. If the framing on $K_0$ is $-1$, the manifold we get by doing surgery on the two leftmost components is $P(1,0)$ which has $b_1>0$. Thus the slope $-1$ on $K_0$ is rationally null-homologous in $P(1,1)\setminus K_0$.
As a result, the square of the generator of $H_2(W_1,P(1,1))$ corresponding to the two-handle attached along $K_0$ is $-a_0$. So the intersection lattices on $(W,P(2,1))$ and $(W_1,P(1,1))$ are isomorphic.

Similarly, we see that the square of the generator of $H_2(W',P(0,1))$ and $H_2(W_1',P(1,1))$ corresponding to the two-handle attached along the knot with label $-a_0'$ is $-(a_0'-1)$.
So the intersection lattices are isomorphic.
\end{proof}

\begin{lemma}
All four cobordisms $W,W_1,W',W_1'$ satisfy the assumptions in the beginning of Section~\ref{sec:cobordism}.
\end{lemma}
\begin{proof}
The cobordism $W$ satisfies the assumptions by its construction.

For $W_1,W_1'$, notice that $P(1,1)$ bounds a rational homology ball $Z_1$ with $H_1(Z_1)\cong\mathbb Z/2\mathbb Z$. Since $H_1(P(1,1))$ is cyclic, the kernel of the surjective map $H_1(P(1,1))\to H_1(Z_1)$ is $2H_1(P(1,1))$. From Figures~\ref{fig:Kirby4}~and~\ref{fig:Kirby6}, we see that the knot with label $-(a_0+1)$ or $-a_0'$ represents an element in $2H_1(P(1,1))$. So $W_1,W_1'$ satisfy the assumptions.

For $W'$, the rational ball bounded by $\mathbb RP^3\#\mathbb RP^3$ is $Z_0=(\mathbb RP^3\setminus B^3)\times I$. Clearly, the knot labeled with $-a_0'$ in Figure~\ref{fig:Kirby5} is null-homologous in $Z_0$.
\end{proof}

\subsection{The proof of Proposition~\ref{prop:SharpCobordism}}

Recall from Section~\ref{subsect:Kirby} that $P(a,1)$ bounds a rational homology ball $Z_a$ with $H_1(Z_a)\cong\mathbb Z/2\mathbb Z$. There are exactly two Spin$^c$ structures $\mathfrak e_0,\mathfrak e_1\in\Spin^\C(P(a,1))$ which extend over $Z_a$.
Let $\mathfrak o_0,\mathfrak o_1\in\Spin^\C(P(a,1))$ be two other Spin$^c$ structures, such that $d(P(a,1),\mathfrak o_1)\ge d(P(a,1),\mathfrak o_0)$.

\begin{lemma}\label{lem:CorrPa1}
The correction terms of $P(a,1)$ are
\[
d(P(a,1),\mathfrak e_0)=d(P(a,1),\mathfrak e_1)=0,\]
\[
 d(P(a,1),\mathfrak o_0)=-\frac{a+2}4,\quad d(P(a,1),\mathfrak o_1)=-\frac{a-2}4.
\]
\end{lemma}
\begin{proof}
The correction terms of $P(a,1)$ are computed in \cite[Example~15]{Doig1}, and they are $\{0,0,-\frac{a+2}4,\frac{a-2}4\}$.
It is a standard fact that $d(P(a,1),\mathfrak e_i)=0, i=0,1$ \cite[Proposition~9.9]{OSzAbGr}. So we must have $d(P(a,1),\mathfrak o_i)=-\frac{a+2}4+i, i=0,1$, by our choice of $\mathfrak o_0,\mathfrak o_1$.
\end{proof}

\begin{proof}[Proof of Proposition~\ref{prop:SharpCobordism} in the case $a_{-1}=2$]
By \cite[Theorem~9.6]{OSzAbGr},
\begin{equation}\label{eq:CorrDW}
d(P(p,q),\mathfrak t)\ge D_W(P(p,q),\mathfrak t).
\end{equation}
Also, since $W_1,W',W_1'$ are sharp, we have
\begin{eqnarray*}
d(P(p-q,q),\mathfrak t_1)&=& D_{W_1}(P(p-q,q),\mathfrak t_1),\\
d(P(p,-q),\mathfrak t)&=& D_{W'}(P(p,-q),\mathfrak t)\\
d(P(p-q,-q),\mathfrak t_1)&=& D_{W_1'}(P(p-q,-q),\mathfrak t_1).
\end{eqnarray*}

By Corollary~\ref{cor:SumInv}, Lemma~\ref{lem:IsomH2} and Lemma~\ref{lem:CorrPa1},
\begin{eqnarray*}
\sum_{\mathfrak t\in\Spin^\C(P(p,q))}D_W(P(p,q),\mathfrak t)&=&-\frac{2q}4+\sum_{\mathfrak t_1\in\Spin^\C(P(p-q,q))}D_{W_1}(P(p-q,q),\mathfrak t_1),\\
-\frac{2q}4+\sum_{\mathfrak t\in\Spin^\C(P(p,-q))}D_{W'}(P(p,-q),\mathfrak t)&=&\sum_{\mathfrak t_1\in\Spin^\C(P(p-q,-q))}D_{W_1'}(P(p-q,-q),\mathfrak t_1).
\end{eqnarray*}
Adding the above two equalities together, and using (\ref{eq:CorrDW}) and the three equalities after it, we get
\begin{eqnarray*}
0&=&\sum_{\mathfrak t\in\Spin^\C(P(p,q))}d(P(p,q),\mathfrak t)+\sum_{\mathfrak t\in\Spin^\C(P(p,-q))}d(P(p,-q),\mathfrak t)\\
&\ge&\sum_{\mathfrak t\in\Spin^\C(P(p,q))}D_W(P(p,q),\mathfrak t)+\sum_{\mathfrak t\in\Spin^\C(P(p,-q))}D_{W'}(P(p,-q),\mathfrak t)\\
&=&\sum_{\mathfrak t_1\in\Spin^\C(P(p-q,q))}D_{W_1}(P(p-q,q),\mathfrak t_1)+\sum_{\mathfrak t_1\in\Spin^\C(P(p-q,-q))}D_{W_1'}(P(p-q,-q),\mathfrak t_1)\\
&=&\sum_{\mathfrak t_1\in\Spin^\C(P(p-q,q))}d(P(p-q,q),\mathfrak t_1)+\sum_{\mathfrak t_1\in\Spin^\C(P(p-q,-q))}d(P(p-q,-q),\mathfrak t_1)\\
&=&0.
\end{eqnarray*}
So the equality in (\ref{eq:CorrDW}) must hold.
\end{proof}

%
\section{The changemaker condition when $q < p<2q$}\label{sec:changemaker}


\subsection{Positive definite manifold with boundary $P(2,1)$}

The goal of this subsection is to prove the following proposition.
\begin{prop}\label{D4}
If $X$ is a positive definite, simply connected four-manifold with $\partial X \cong P(2,1)$, then the intersection form of $X$ is isomorphic to $D_4 \oplus \Z^{n-4}$ for some $n$.
\end{prop}

\begin{lemma}\label{evensublattice}
If $L \subset \Z^n$ is an index--two sublattice, then $L \cong D_k \oplus \Z^{n-k}$ for some $k \ge 1$. (In fact, there are indices $i_1,\dots,i_k$ such that $L$ contains exactly the elements of $\Z^n$ that have even pairing with $e_{i_1} + \cdots + e_{i_k}$.) 
There are always two elements $x \in \overline L$ with $b(x,x) = 0 \pmod{1}$, and the other two elements satisfy $b(x,x) = k/4 \pmod{1}$.
\end{lemma}
\begin{proof}
Let $L \subset \Z^n$ have index two, and let $i_1,\dots,i_k$ be an enumeration of the indices $i$ for which $e_i \not \in L$. Since $L$ has index two, the elements $\pm e_{i_j} \pm e_{i_{j'}}$ are all in $L$. Since these elements generate $D_k$, we have $L \cong D_k \oplus \Z^{n-k}$. 

The dual lattice $L^*$ is the set of elements of $\Q^n$ with integral inner product with each element of $L$, and in this representation we have that $L^*$ is the set of vectors with integer components in all entries other than $i_1,\dots,i_k$, and with the components in entries $i_1,\dots,i_k$ either all integers or all half integers. Therefore, the discriminant group $\overline L$ can be represented by the four vectors $0$, $z = e_{i_1}$, and
\begin{align*}
a &= \frac12\left(e_{i_1} + e_{i_2} + \cdots + e_{i_k}\right), \\
b &= \frac12\left(-e_{i_1} + e_{i_2} + \cdots + e_{i_k}\right).
\end{align*}
We have $\braket{z}{z} = 1 \equiv 0 \pmod{1}$, and $\braket{a}{a} = \braket{b}{b} = k/4$.
\end{proof}

\begin{lemma}\label{dinvDk}
The d-invariant of $L = D_{k} \oplus \Z^{n-k}$ takes on the values $0,0,-k/4, 1 - k/4$.
\end{lemma}
\begin{proof}
The d-invariant is invariant under stable isomorphisms, so we can assume $L = D_k$. Then a set of short representatives of the classes of characteristic covectors is $(1,\dots,1)$, $(-1,1,\dots,1)$, $(0,\dots,0)$, and $(2,0,\dots,0)$. These have norms $k$, $k$, $0$, and $4$. The result now follows: see Equation~\eqref{eq:d}.
\end{proof}

\begin{proof}[Proof of Proposition~\ref{D4}]
As in Section~\ref{subsect:Kirby}, $P(2,1)$ bounds a rational homology ball $Z_2$ with \[H_1(Z_2) \cong \Z/2\mathbb Z, H_2(Z_2)=0.\] If $X$ is any simply connected positive definite 4-manifold with boundary $P(2,1)$, then $\widehat X:=X \cup_{P(2,1)} (-Z_2)$ is a closed, positive definite 4-manifold. Since $\widehat X$ can be obtained from $X$ by attaching a two-handle, a three-handle and a four-handle, $\widehat X$ is also simply connected. By \cite{Donaldson1983}, $\widehat X$ has intersection form $\Z^n$.

In the long exact sequence for the pair $(\widehat X, X)$, we have
\begin{equation*}
H_3(\widehat X, X) \to H_2(X) \to H_2(\widehat X) \to H_2(\widehat X , X) \to H_1(X).
\end{equation*}
We have $H_3(\widehat X, X)\cong H_3(Z_2,\partial Z_2)\cong H^1(Z_2)=0$, $H_2(\widehat X , X)\cong H^2(Z_2)\cong\Z/2\mathbb Z$, $H_1(X)=0$, and both $H_2(X)$ and $H_2(\widehat X)$ are torsionfree. Therefore, we have a short exact sequence
\[
0 \to H_2(X) \to H_2(\widehat X ) \to \Z/2\mathbb Z \to 0,
\]
so $H_2(X)$ is an index-two subgroup of $H_2(\widehat X )$ under the natural inclusion map. Since $\widehat X$ has intersection lattice $\Z^n$, the intersection lattice of $X$ is an index-two sublattice of $\Z^n$, so, by Lemma~\ref{evensublattice}, is isomorphic to $D_k \oplus \Z^{n-k}$.

Let $X_0$ be the positive definite plumbing $4$-manifold with intersection form $D_4$, then $P(2,1) = \partial X_0$.
Since the discriminant group and linking pairing of the intersection form of a 4-manifold are invariants of its boundary, Lemma~\ref{evensublattice} implies that $k$ must be divisible by $4$. Since the d-invariant of the intersection form of a positive definite 4-manifold gives an upper bound on the d-invariant of its boundary \cite{OSzAbGr} and $-X_0$ is sharp \cite{OSzPlumbed}, Lemma~\ref{dinvDk} implies that $k \le 4$. Therefore, $k = 4$, and the result follows.
\end{proof}

\begin{cor}\label{sharp}
Any negative definite, simply connected $4$-manifold with boundary $-P(2,1)$ is sharp.
\end{cor}
\begin{proof}
The $4$--manifold $-X_0$ is sharp. By Proposition~\ref{D4}, any negative definite, simply connected 4-manifold with boundary $-P(2,1)$ has the same intersection form as that of $-X_0\#(n-4)\overline{\mathbb{C}P^2}$.
\end{proof}


\subsection{The changemaker condition}

Whenever $q<p < 2q$, using Proposition~\ref{prop:SharpCobordism}, there is a sharp cobordism $W$ from $P(2,1)$ to $P(p,q)$. Suppose $P(p,q)$ is positive surgery on some knot $K\subset S^3$. Let $X = W \cup_{P(p,q)} (-W_{4q}(K))$, then $X$ is a  negative definite manifold with boundary $-P(2,1)$. Since $X$ is obtained from $W_{4q}$ (which is simply connected) by adding two-handles, $X$ is simply connected.
By combining Corollary~\ref{sharp} and Proposition~\ref{D4}, $X$ is sharp and has intersection lattice $-(D_4 \oplus \Z^{n-2})$.  Also, for $Z_2$ the rational homology ball with boundary $P(2,1)$, the manifold $\widehat X=X \cup_{P(2,1)}(-Z_2)$ is closed, simply connected and negative definite, so has intersection lattice $-\Z^{n+2}$. From Kirby diagrams for $W$ and $Z = W \cup_{P(2,1)}(-Z_2)$ (see Figure~\ref{fig:Kirby3}), we can also see that the intersection lattice of $Z$ is the linear lattice $\Lambda(q,-p)$ with vertex basis $x_0,\dots,x_n$, and the intersection lattice of $W$ is (as a sublattice of $\Lambda(q,-p)$) spanned by $2x_0,x_1,\dots,x_n$. Therefore, the following diagram of homology groups
\begin{equation*}
\xymatrix{
H_2(W) \ar[r] \ar[d] & H_2(Z) \ar[d] \\
H_2(X) \ar[r]  & H_2(\widehat X)
}
\end{equation*}
with maps induced by inclusions is isomorphic to the diagram
\begin{equation*}
\xymatrix{
\langle 2x_0, x_1,\dots,x_n \rangle \ar[r] \ar[d] & \langle x_0, x_1,\dots,x_n \rangle = -\Lambda(q,-p) \ar[d] \\
-(D_4 \oplus \Z^{n-2}) \ar[r]  & -\Z^{n+2}.
}
\end{equation*}

\begin{lemma}\label{lem:ZXintersection}
Regarding $H_2(W)$ as subgroups of $H_2(Z)$ and $H_2(X)$, which are subgroups of $H_2(\widehat X)$, then
\[H_2(W)=H_2(Z)\cap H_2(X).\]
\end{lemma}
\begin{proof}
By the exact sequence $H_2(Z)\to H_2(\widehat X)\to H_2(\widehat X,Z)$, an element $\beta\in H_2(\widehat X)$ is contained in the image of $H_2(Z)$ if and only if the image of $\beta$ in $H_2(\widehat X,Z)\cong H_2(W_{4q}(K),\partial W_{4q}(K))$ is zero. Similarly, $\beta$ is contained in the image of $H_2(X)$ if and only if the image of $\beta$ in $H_2(\widehat X,X)\cong H_2(Z_2,\partial Z_2)$ is zero, and $\beta$ is contained in the image of $H_2(W)$ if and only if the image of $\beta$ in $H_2(\widehat X,W)\cong H_2(Z_2,\partial Z_2)\oplus H_2(W_{4q}(K),\partial W_{4q}(K))$ is zero. Our conclusion follows easily.
\end{proof}

The last piece of data we need is the class $[\widehat F] \in H_2(-W_{4q}(K)) \subset H_2(X)$, where $\widehat F$ is obtained by smoothly gluing the core of the handle attachment to a copy of a minimal genus Seifert surface $F$ for $K$; its homology class generates the second homology. Note that $H_2(-W_{4q}(K))$ is orthogonal to all of $H_2(W)$ and satisfies $\braket{[\widehat F]}{[\widehat F]} = -4q$ since $-W_{4q}(K)$ is negative definite. Let $$\phi: \Z/4q\Z \to \Spin^\C(P(p,q))$$ be the correspondence with $\phi(i)$ equal $\mathfrak s_0|_{P(p,q)}$ for $\mathfrak s_0$ any $\Spin^\C$ structure on $-W_{4q}(K)$ satisfying
\begin{equation*}
\braket{c_1(\mathfrak s_0)}{[\widehat F]} \equiv -4q + 2i \pmod{8q}.
\end{equation*}

\begin{prop}\label{protochangemaker}
There is an extension $\mathfrak r\in \Spin^\C(X)$ of $\phi(i)$ over $X$ with $c_1(\mathfrak r)$ a short characteristic covector of $D_4 \oplus \Z^{n-2}$ if any only if $g(K) \le i \le 4q - g(K)$.
\end{prop}
\begin{proof}
Since $X$ has boundary $-P(2,1)$ and $b_2(X) = n+2$, we have that for any $\mathfrak r \in \Spin^\C(X)$,
\begin{equation}\label{overallbound}
d(-P(2,1),\mathfrak r|_{P(2,1)}) \ge \frac{(c_1(\mathfrak r))^2 + (n+2)}{4},
\end{equation}
and since $X$ is sharp this is an equality if and only if $c_1(\mathfrak r)$ is a short characteristic covector of $-H_2(X) = D_4 \oplus \Z^{n-2}$. Similarly, for any $\mathfrak s_1 \in \Spin^\C(W)$,
\begin{equation}\label{bottombound}
d(P(p,q),\mathfrak s_1|_{P(p,q)}) \ge d(P(2,1),\mathfrak s_1|_{P(2,1)}) + \frac{(c_1(\mathfrak s_1))^2 + (n+1)}{4}
\end{equation}
and since $W$ is sharp as a cobordism, for each $\mathfrak t \in \Spin^\C(P(p,q))$ there is some $\mathfrak s_1 \in \Spin^\C(W)$ such that this is an equality and $\mathfrak s_1|_{P(p,q)} = \mathfrak t$. 

For $\mathfrak s_0 \in \Spin^\C(-W_{4q}(K))$ with $$\braket{c_1(\mathfrak s_0)}{[\widehat F]} = -4q + 2i$$ (so that in particular $\phi(i) = \mathfrak s_0|_{P(p,q)}$), we have
\[
(c_1(\mathfrak s_0))^2=-\frac{(-4q+2i)^2}{4q}.
\]
Using (\ref{eq:pCorr}) and (\ref{eq:nSurgCorr}), we have
\begin{equation*}
d(P(p,q), \mathfrak s_0|_{P(p,q)}) = \frac{-(c_1(\mathfrak s_0))^2 - 1}{4} - 2t_{\min\{i,4q-i\}}(K).
\end{equation*}
Since $t_i(K) \ge 0$ and (\ref{eq:ti=0}),
\begin{equation}\label{topbound}
d(P(p,q), \mathfrak s_0|_{P(p,q)}) \le \frac{-(c_1(\mathfrak s_0))^2 - 1}{4}
\end{equation}
with equality if and only if $\braket{c_1(\mathfrak s_0)}{[\widehat F]} = -4q + 2i$ for some $i$ with $g(K) \le i \le 4q - g(K)$. Note that inequality~\eqref{overallbound} is the difference of inequalities~\eqref{topbound} and~\eqref{bottombound} if $\mathfrak s_0|_{P(p,q)}=\mathfrak s_1|_{P(p,q)}$. If $g(K) \le i \le 4q - g(K)$, then there is some extension $\mathfrak s_0$ of $\phi(i)$ over $-W_{4q}(K)$ that achieves equality in~\eqref{topbound}, and there is always some extension $s_1$ of $\phi(i)$ over $W$ achieving equality in~\eqref{bottombound}. These two $\Spin^\C$ structures glue to a $\Spin^\C$ structure $\mathfrak r$ on $X = W \cup (-W_{4q}(K))$ that will achieve equality in~\eqref{overallbound}, so $c_1(\mathfrak r)$ is short and $\mathfrak r|_{P(p,q)} = \phi(i)$.

Conversely, if $\mathfrak r \in \Spin^\C(X)$ has $c_1(\mathfrak r)$ short, then $\mathfrak r$ achieves equality in~\eqref{overallbound}, so $\mathfrak s_0 = \mathfrak r|_{-W_{4q}(K)}$ and $\mathfrak s_1 = \mathfrak r|_{W}$ will achieve equality in~\eqref{bottombound}~and~\eqref{topbound}, respectively. Therefore, $\mathfrak s_0|_{P(p,q)} = \mathfrak r|_{P(p,q)}$ will equal $\phi(i)$ for some $g(K) \le i \le 4q - g(K)$.
\end{proof}


Putting all of these together, we have a Euclidean lattice $\Z^{n+2} = -H_2(\widehat X)$, with a corank--$1$, linear sublattice $$-H_2(W) \cong \Lambda(q,-p) = \langle x_0,\dots,x_n \rangle$$ and a sublattice $D_4 \oplus \Z^{n-2} = -H_2(X)$ such that
\begin{equation}\label{eq:Intersect}
\langle 2x_0,\dots,x_n \rangle = \langle x_0,\dots,x_n \rangle\cap (D_4 \oplus \Z^{n-2}).
\end{equation}

Since $\Lambda(q,-p)$ has discriminant $q$ and corank $1$ and is embedded primitively in $\Z^{n+2}$ (this follows from the long exact sequence of the pair $(X \cup Z_0, W \cup Z_0)$), the orthogonal complement of $\Lambda(q,-p)$ has discriminant $q$ and rank $1$, so is generated by a vector $\sigma$ with $\braket{\sigma}{\sigma} = q$.
Since $|\braket{[\widehat F]}{[\widehat F]}| = 4q$ and $[\widehat F]$ is contained in the orthogonal complement of $\Lambda(q,-p)$, we must have $[\widehat F] = 2\sigma$. Therefore, Proposition~\ref{protochangemaker} gives the following:

\begin{prop}\label{changemaker}
If $P(p,q)$ is the result of $4q$ surgery on some knot $K \subset S^3$ and $q<p < 2q$, then there is an embedding of $\Lambda(q,-p)$ into $\Z^{n+2}$ as the orthogonal complement of a vector $\sigma$ and an embedding $D_4 \oplus \Z^{n-2} \into \Z^{n+2}$ such that there exists some short characteristic covector $\chi$ for $D_4 \oplus \Z^{n-2}$ with $\braket{\chi}{\sigma} = i$ if and only if $-2q + g(K) \le i \le 2q-g(K)$.
\end{prop}

Pushing the logic of Proposition~\ref{protochangemaker} a little further, the Alexander polynomial of $K$ can be recovered from $\sigma$:
\begin{prop}\label{prop:AlexanderComputation}
For $0 \le i \le 2q$, the torsion coefficient $t_i(K)$ satisfies
\begin{equation*}
t_i(K) = \min_{\substack{\chi \in \operatorname{Char}(D_4 \oplus \Z^{n-2}) \\ \braket{\chi}{\sigma} = 2q-i}} \left\lceil \frac{\braket{\chi}{\chi} - n - 2}{8} \right\rceil.
\end{equation*}
\end{prop}
\begin{proof}
Since $[\widehat{F}] = 2\sigma$ and the intersection lattice on $X$ is $D_4 \oplus \Z^{n-2}$, any characteristic covector $\chi$ for $D_4 \oplus \Z^{n-2}$ with $\braket{\chi}{\sigma}  = 2q-i$ is the first Chern class of a $\Spin^\C$ structure $\mathfrak{r}$ on $X$ with 
\begin{equation}\label{eq:rOnF}
\braket{c_1(\mathfrak{r})}{[\widehat{F}]} = -4q + 2i.
\end{equation}
 (Note that we need to change the sign of the inner product.) Then, exactly as in the proof of Proposition~\ref{protochangemaker}, the restriction of $\mathfrak{r}$ to $-W_{4q}=-W_{4q}(K)$ satisfies
\begin{equation}\label{torsionterm}
d(P(p,q), \mathfrak{r}|_{P(p,q)}) = \frac{-(c_1(\mathfrak{r}|_{-W_{4q}}))^2 - 1}{4} - 2t_i(K).
\end{equation}
Let $\mathfrak s_1$ be the restriction of $\mathfrak{r}$ to $W$, then $\mathfrak s_1$ satisfies
\begin{equation}\label{topbound2}
d(P(p,q), \mathfrak{s}_1|_{P(p,q)}) \ge d(P(2,1),\mathfrak{s}_1|_{P(2,1)}) + \frac{(c_1(\mathfrak{s}_1))^2 + (n+1)}{4}
\end{equation}
Combining~\eqref{torsionterm}~and~\eqref{topbound2} together,
\begin{equation}\label{torsionbound}
t_i(K) \le \frac{-(c_1(\mathfrak{r}))^2 - (n + 2)}{8} - \frac{d(P(2,1),\mathfrak{r}|_{P(2,1)})}{2}.
\end{equation}

Using Proposition~\ref{prop:SharpCobordism}, some $\mathfrak{s}_1\in\Spin^\C(W)$ achieves equality in~\eqref{topbound2} with $\mathfrak{s}_1|_{P(p,q)}=\varphi(i)$. Let $\mathfrak{r}\in\Spin^\C(X)$ be the extension of $\mathfrak s_1$ with (\ref{eq:rOnF}), then $\mathfrak{r}$ achieves equality in~\eqref{torsionbound}. Therefore,
\begin{equation}
t_i(K) = \min_{\substack{\mathfrak{r} \in \Spin^\C(X) \\ \braket{c_1(\mathfrak{r})}{[\widehat{F}]} = -4q + 2i}} \frac{-(c_1(\mathfrak{r}))^2 - (n + 2)}{8} - \frac{d(P(2,1),\mathfrak{r}|_{P(2,1)})}{2}
\end{equation}
Since $t_i(K)$ is an integer and $d(P(2,1),\mathfrak{r}|_{P(2,1)})$ will always be either $0$ or $-1$, we get
\begin{equation}
t_i(K) = \min_{\substack{\mathfrak{r} \in \Spin^\C(X) \\ \braket{c_1(\mathfrak{r})}{[\widehat{F}]} = -4q + 2i}} \left\lceil \frac{-(c_1(\mathfrak{r}))^2 - (n + 2)}{8} \right\rceil.
\end{equation}

Finally, $\Spin^\C$ structures $\mathfrak{r}$ on $X$ with (\ref{eq:rOnF}) correspond (under the first Chern class and a change in the sign of the inner product) with characteristic covectors $\chi$ of $D_4 \oplus \Z^{n-2}$ with $\braket{\chi}{\sigma} = 2q - i$, and $-(c_1(\mathfrak{r}))^2 = \braket{\chi}{\chi}$, so the desired formula follows.
\end{proof}


By Proposition~\ref{D4},
specifying a sublattice $D_4 \oplus \Z^{n-2} \subset \Z^{n+2}$ is equivalent to choosing $4$ indicies $a > b > c > d$ such that for $v \in \Z^{n+2}$, $v \in D_4 \oplus \Z^{n-2}$ if and only if $\braket{v}{e_a + e_b + e_c + e_d}$ is even. The characteristic covectors for $D_4 \oplus \Z^{n-2}$ come in two types: those that are the restrictions of characteristic covectors of $\Z^{n+2}$, which can be represented by elements of $\Z^{n+2}$ with all entries odd, and those that are not, which can be represented by elements of $\Z^{n+2}$ with the entries in positions $a,b,c,$ and $d$ even and all other entries odd. Call these two types of covectors \textit{even} and \textit{odd}, respectively. The short characteristic covectors are exactly the ones with all odd entries equal to $\pm 1$, and the even entries (if any) equal to $\pm 2,0,0,$ and $0$ in some order.

As in \cite{greene:LSRP}, we will assume $\sigma=(\sigma_0,\sigma_1,\dots,\sigma_{n+1})$ with
\[
0\le \sigma_0\le\sigma_1\le\dots\le\sigma_{n+1}.
\]
Moreover, we can assume that for any two indices $i,j\in\{0,1,\dots,n+1\}$, we always have
\begin{equation}\label{eq:abcdMAX}
i>j, \qquad\text{if }\sigma_i=\sigma_j, i\in\{a,b,c,d\}, \text{and } j\notin\{a,b,c,d\}.
\end{equation}

\begin{definition}
Let $\Sh(D_4 \oplus \Z^{n-2}) = \Sh_0 \cup \Sh_1$, with $\Sh_0 = \Sh(\Z^n)$ the set of even short characteristic covectors and $\Sh_1 = \Sh(D_4 \oplus \Z^{n-2}) - \Sh_0$ the set of odd characteristic covectors. Let $$\chi^0 = -\sum_{i = 0}^{n+3} e_i$$ and $$\chi^1 = -2e_a - \sum_{i \not \in \{a,b,c,d\}} e_i$$ be the elements of $\Sh_0$ and $\Sh_1$, respectively, minimizing $\braket{\chi}{\sigma}$. Let
$$\mathcal{T}_0 = \left\{\left.\frac12 (\chi - \chi^0)\right|\chi \in \Sh_0\right\}$$
and $$\mathcal{T}_1 = \left\{\left.\frac12 (\chi - \chi^1)\right|\chi \in \Sh_1\right\}$$ be called the sets of even and odd test vectors, respectively.
\end{definition}

For $\chi \in \mathbb Z^{n+2}$, let $\chi_i$ denote the component of $\chi$ corresponding to the index $i$.
The following result is easy to see.

{\prop\label{prop:Sh1}For $\chi\in \mathcal{T}_1$, $(\chi_d, \chi_c,\chi_b,\chi_a)=(\pm 1, 0, 0, 1)$ or $(0, \pm 1, 0, 1)$ or $(0, 0, \pm 1, 1)$ or $(0,0,0,2)$ or $(0,0,0,0)$.}

\begin{prop}\label{prop:T0T1}
The sets $\{\braket{\chi}{\sigma}~|~\chi \in \mathcal{T}_0\}$ and $\{\braket{\chi}{\sigma}~|~\chi \in \mathcal{T}_1\}$ are both intervals of integers beginning at $0$. Also,
\begin{equation}\label{eq:T0T1}
\sum_{i=0}^{n+1}\sigma_i=\max\{\braket{\chi}{\sigma}~|~\chi \in \mathcal{T}_0\}=\max\{\braket{\chi}{\sigma}~|~\chi \in \mathcal{T}_1\}\pm 1.
\end{equation}
 \end{prop}
\begin{proof}
By Proposition~\ref{changemaker}, the set $\{\braket{\chi}{\sigma}~|~\chi \in \Sh(D_4 \oplus \Z^{n-2})\}$ is an interval of integers. For each $i\in\{0,1\}$, the set $\{\braket{\chi}{\sigma}~|~\chi \in \Sh_i\}$ contains the elements of this interval with the same parity. So the parities are different for $i=0$ and $i=1$. In particular, both sets are arithmetic progressions of step size $2$, so subtracting off the smallest element and dividing by $2$ gives intervals beginning at $0$.
\end{proof}

\begin{cor}\label{cor:changemaker}
$\sigma$ is a changemaker.
\end{cor}
\begin{proof}
The set $\mathcal{T}_0$ consists of just vectors with all entries $0$ or $1$.
\end{proof}

\begin{proof}[Proof of Theorem~\ref{thm:restriction}]
This follows from the combination of Corollary~\ref{cor:changemaker} and Proposition~\ref{changemaker}.
\end{proof}

{\cor\label{cor:sigma} $\sigma_a=\sigma_b+\sigma_c+\sigma_d+\theta$, where $\theta\in\{-1,1\}$.}
\begin{proof}
Using~\eqref{eq:T0T1}, we see that
\[
\sum_{i=0}^{n+1}\sigma_i = 2e_a+\Big(\sum_{j\not \in \{a, b, c, d\}}\sigma_j\Big) \pm 1.
\]
The result is now immediate.
\end{proof}

\begin{lemma}\label{lem:OddPairing}
An irreducible vector $v\in\sigma^{\perp}$ has an odd pairing with the vector $e_a+e_b+e_c+e_d$ if and only if $[v]$ contains $x_0$.
\end{lemma}
\begin{proof}
Suppose $v\in\sigma^{\perp}$ is irreducible.
The pairing $\braket{v}{e_a+e_b+e_c+e_d}$ is even if and only if $v\in D_4 \oplus \Z^{n-2}$, which is equivalent to $v\in \langle 2x_0,\dots,x_n \rangle$ by (\ref{eq:Intersect}). Since $v$ is irreducible, $v\notin \langle 2x_0,\dots,x_n \rangle$ if and only if $[v]$ contains $x_0$.
\end{proof}

Let
\begin{equation}\label{eq:Gdef}
G = 1 + \sigma_0 + \sigma_1 + \cdots + \sigma_{d-1}.
\end{equation}

\begin{lemma}\label{lem:Gbound}There exists $\chi\in \T_1$ with $\braket{\chi}{\sigma}=G$.
Let $f$ be the minimal index such that $f>d$ and $f\notin\{a,b,c\}$.
\newline If $\chi_a=0$, then \[G\ge\sigma_f.\]
If  $\chi_a \ne 0$, then
\[G\ge \sigma_a-\sigma_b=\sigma_c+\sigma_d+\theta.\]
\end{lemma}
\begin{proof}
Using Proposition~\ref{prop:T0T1}, there exists $\chi \in \T_1$ with $\braket{\chi}{\sigma}=G$.
If $\chi_a=0$, by Proposition~\ref{prop:Sh1} we have $\chi_b=\chi_c=\chi_d=0$, then there must be an index $i>d$, $i\notin\{a,b,c\}$, with $\chi_i\ne 0$ as otherwise $\braket{\chi}{\sigma}<G$.
So \[G=\braket{\chi}{\sigma}\ge\sigma_i\ge\sigma_f.\]
If $\chi_a \ne 0$, by Proposition~\ref{prop:Sh1} we have
\[
G=\braket{\chi}{\sigma}\ge \sigma_a-\sigma_b= \sigma_c+\sigma_d+\theta.\qedhere
\]
\end{proof}

\section{Bounding $d$}\label{sec:dbounding}

In this section, we will prove that $d=0$. We assume that $d>0$ for contradiction.

Recall that we write $(e_0,e_1, \dots, e_{n+1})$ for the orthonormal basis of $\Z^{n+2}$, and $\sigma = \sum_i \sigma_i e_i$.
Since $\Lambda(q,-p)$ is indecomposable (Proposition~\ref{indecomposable}), $\sigma_0\ne0$, otherwise $\sigma^\perp$ would have a direct summand $\mathbb Z$. So $\sigma_0=1$.  By Lemma~\ref{lem:OddPairing}, we have that $[v_d]$ contains $x_0$. Set
\begin{equation}\label{eq:w}
w=\theta e_0+e_d+e_c+e_b-e_a,
\end{equation}
where $\theta\in\{-1,1\}$ is as in Corollary~\ref{cor:sigma}. 

{\lemma\label{lemma:w}$w$ is an irreducible vector of $\sigma^\perp$. Also, $x_0 \not \in [w]$.}
\begin{proof}
Corollary~\ref{cor:sigma} shows that $w$ is in $\sigma^\perp$. Suppose $w=x+y$ with $x,y\in \sigma^\perp$ and $\braket{x}{y}\ge 0$.
If both $x,y$ are nonzero, by Lemma~\ref{lem:Decomp} we may assume that one of the vectors is $e_d-e_0$ and the other is $-e_a+e_b+e_c$. Both vectors will then be irreducible and $x_0\in [x],[y]$. That implies $\braket{x}{y}\ne 0$, which is a contradiction. The second statement is immediate from Lemma~\ref{lem:OddPairing}.
\end{proof}

\begin{cor}\label{cor:epsilon+1}
If one of the following two conditions holds, then $\theta=1$:
\newline(1) $\sigma_d=1$;
\newline(2) there exists a vector $v$ with $\braket{v}{e_0}=-\braket{v}{e_d}=1$, $\max\supp(v)=d$ and $|\braket{v}{w}|\le1$.
\end{cor}
\begin{proof}
If $\sigma_d=1$ and $\theta=-1$, then $w=(-e_0+e_d)+(e_c+e_b-e_a)$ is reducible, a contradiction to Lemma~\ref{lemma:w}.

If there exists a vector $v$ as in the statement, then since $\braket{v}{e_0}=-\braket{v}{e_d}=1$ and $\max\supp(v)=d$, we have $\braket{v}{w}=\theta-1$. Using $|\braket{v}{w}|\le1$, we have $\theta=1$.
\end{proof}

{\rmk\label{wpairing}When $d>0$, we have $[v_d]$ contains $x_0$. For any $0<i<d$, $[v_i]$ does not contain $x_0$. Also, $\supp(v_i)\cap\supp(w)=\emptyset \text{ or } \{0\}$, so $|\braket{w}{v_i}|\le 2$.}

\begin{lemma}\label{lem:wdagger}
Suppose that $0\notin \supp(v_d)$, then $[v_d]\dagger[w]$.
\end{lemma}
\begin{proof}
We can compute $\braket{w}{v_d}=-1$.
Assume that $[v_d]\dagger[w]$ does not happen, then either $[v_d]\prec[w]$ or $[v_d]\pitchfork[w]$. Note that $x_0\in[v_d]$ and $x_0\notin[w]$.

If $[v_d]\prec[w]$, then $|v_d|=2$, and $[w]$ and $[v_d]$ share their right end. This is not possible since $|w|>|v_d|$.

If $[v_d]\pitchfork[w]$, then $|[v_d]\cap[w]|=3$, and there exists $\epsilon\in\{-1,1\}$ such that $w=\epsilon[w]$ and $v_d=-\epsilon[v_d]$. So $w+v_d=x+y$ with $[x]$ and $[y]$ being distant, and we may assume $x_0\in[x]$.
Since $v_d$ is not tight, $v_d$ is unbreakable. So $|v_d|=|[w]\cap[v_d]|=3$, and $|x|=2$.
We get $v_d=e_i+e_{d-1}-e_d$ for some $0<i<d-1$, and
\[
w+v_d=\theta e_0+e_i+e_{d-1}+e_c+e_b-e_a.
\]
Using Lemma~\ref{lem:Decomp} and the fact that $x_0\in[x]$, we have either $x=e_j-e_a$ for some $j\in\{0,i,d-1\}$ or $x=-e_0+e_k$ for some $k\in\{c,b\}$.
If $x=e_j-e_a$, then $\sigma_j=\sigma_a=\sigma_b$, contradicting Corollary~\ref{cor:sigma}.
If $x=-e_0+e_k$, then $\theta=-1$ and $\sigma_d=\sigma_k=1$, contradicting Corollary~\ref{cor:epsilon+1}.
\end{proof}

\begin{lemma}\label{lem:ThreeCases}
Suppose that $0\notin\supp(v_d)$ and $|\braket{v_i}{v_d}|=1$ for some $i$ with $0<i<d$. Then $i=1$.
\end{lemma}
\begin{proof}
Since $i<d$, $x_0\notin[v_i]$ by Lemma~\ref{lem:OddPairing}. We have $[v_d]\dagger[w]$ by Lemma~\ref{lem:wdagger}.

If $[v_i]\dagger[v_d]$, then $[v_i]$ and $[w]$ share their left end. If $|v_i|>2$, we have
 $2\le|\braket{v_i}{w}|$,
hence $\braket{v_i}{e_0}=2$ and $v_i$ is tight. If we also have $i>1$, then $|v_i|\ge6>|w|$, so
$|\braket{v_i}{w}| = |w|-1=4$, which is not possible. So in order to prove $i=1$, we only need to assume $|v_i|=2$ in this case.

If $[v_i]$ and $[v_d]$ share their right end, then we must have $|v_i|=2$.

In the above two cases we have $|v_i|=2$ and $[v_i]$ abuts the right end of $[v_d]$, so $|\braket{v_i}{w}|=1$, which implies $i=1$.

If $[v_i]\pitchfork[v_d]$, then $|[v_i]\cap[v_d]|=|v_d|=3$. By Lemma~\ref{lem:DistinctHighNorm}, $v_i$ is tight. If $i>1$, $|v_i|\ge6=|w|+|v_d|-2$. Since $[v_d]\dagger[w]$, the interval $[v_i]$ must contain all high weight vertices of $[w]$. Thus $|\braket{w}{v_{i}}|\ge|w|-2=3$, a contradiction (Remark~\ref{wpairing}).
\end{proof}

\begin{lemma}\label{lem:vdgappy}$v_d$ is not gappy.
\end{lemma}
\begin{proof}
Suppose for contradiction that $v_d$ is gappy. Take the index $i$ to be the smallest gappy index of $v_d$. First suppose that $i=0$. Then, using Lemma~\ref{gappy3}, $v_{1}$ will be tight with $|v_1|=5$. Note that $\braket{w}{v_1}=2\theta$, $|v_1|=|w|=5$, so $[w]\pitchfork [v_1]$ with
$|[v_1]\cap[w]|=4$,
and there exists $\epsilon\in\{-1,1\}$ such that $w=\epsilon[w]$ and $v_1=\theta\epsilon[v_1]$.
It follows that $w-\theta v_1
=x+y$ with $[x]$ and $[y]$ being distant, $|x|=|y|=3$.
Now
\[
w-\theta v_1=-\theta e_0+\theta e_1+e_{d}+e_c+e_b-e_a.
\]
Since $x_0\notin[w],[v_1]$, we have $x_0\notin[x],[y]$. Using Lemma~\ref{lem:Decomp}, one of $x,y$ has the form
$\pm e_j+e_k+e_l$, where $j\in\{0,1\}, \{k,l\}\subset\{d,c,b\}$, but this vector is not in $\sigma^{\perp}$, a contradiction.

Suppose $i>0$. Then $i=\min\supp(v_d)$ by \cite[Paragraph~2 in Section~6, and Propositions~8.6, 8.7, 8.8]{greene:LSRP}.
Since $\braket{v_{i+1}}{v_d}=1$, by Lemma~\ref{lem:ThreeCases} we have $i+1=1$, a contradiction.
\end{proof}

{\prop\label{prop:vd} $\min \supp(v_d)\le 1$.}
\begin{proof}
Set $i=\min \supp(v_d)$. If $i>0$, since $\braket{v_{i}}{v_d}=-1$, by Lemma~\ref{lem:ThreeCases} we have $i=1$.
\end{proof}

Let $G$ be defined as in (\ref{eq:Gdef}). Our strategy is to first find a bound for $G$, and then find a bound for the integer $d$. Next, we do a case-by-case analysis to find that indeed $d=0$.

{\lemma\label{lem:vdtight}$v_d$ is not tight.}
\begin{proof}
Suppose for contradiction that $v_d$ is tight. Using Lemma~\ref{lem:Gbound}, we get
\[
\sigma_d=G\ge \min\{\sigma_f,\sigma_d+\sigma_c+\theta\}\ge\min\{\sigma_f,2\sigma_d-1\},
\]
which is not possible by (\ref{eq:abcdMAX}) and  Corollary~\ref{cor:epsilon+1}.
\end{proof}

Combining Proposition~\ref{prop:vd} and Lemmas~\ref{lem:vdgappy}~and~\ref{lem:vdtight}, we have:
{\cor\label{cor:vd}$v_d=v_{d,0}e_0+e_1+\cdots-e_d$ with $v_{d,0}\in \{0,1\}$.}

With the notation of Corollary~\ref{cor:vd} in place, we start the analysis to deduce $d=0$. The following identity will be useful to keep in mind:
\begin{equation}\label{eq:vdG}
\sigma_d = G - 2 + v_{d,0}.
\end{equation}

\begin{lemma}\label{lem:short1}If either $|v_d|>2$ or $d=1$, then
\[
G\ge \sigma_d+\sigma_c+\theta.
\]
\end{lemma}
\begin{proof}
Let $\chi$ be the vector as in Lemma~\ref{lem:Gbound}. By that lemma, it will suffice to show $\chi_a\ne0$. Assume that $\chi_a=0$, then Lemma~\ref{lem:Gbound} implies that $G\ge\sigma_f>\sigma_d$. Using \eqref{eq:vdG}, we have that $G\le\sigma_d+2$, so $\sigma_f\in\{\sigma_d+1,\sigma_d+2\}$.

If $\sigma_f=\sigma_d+1$, set $v'_f=-e_f+e_d+e_0$. If $\sigma_f=\sigma_d+2$, set $v'_f=-e_f+e_d+e_1+e_0$. (Note that $d\not = 1$ in this case, otherwise $G=2\ne\sigma_d+2$.) In either case,
$v'_f$ is irreducible and also in $\sigma^{\perp}$. Since $\braket{v'_f}{e_a+e_b+e_c+e_d}=1$, we get that $x_0\in [v'_f]$. So $[v_d]$ and $[v'_f]$ share their left endpoint. If $|v_d|>2$, then $|\braket{v_d}{v'_f}|\ge2$, which contradicts the direct computation $|\braket{v_d}{v'_f}|\le1$. If $d=1$, using Lemma~\ref{lem:vdtight}, we get $\braket{v_d}{v'_f}=0$: this is still giving a contradiction since the intervals $[v_d]$ and $[v'_f]$ share their left endpoints, and so $\braket{v_d}{v'_f}\not = 0$.
\end{proof}

{\prop\label{prop:Gbound}
If $|v_d|=2$, then either $d=1,G=2$, or else $d=2$, $G\in\{3,4\}$.

If $|v_d|>2$, then $d\in\{3,4\}$, $\theta=-1$, $v_{d,0}=0$, and $1+d\le G\le 5$.}
\begin{proof}
If $|v_d|=2$, our conclusion follows from Corollary~\ref{cor:vd}.

Now we assume that $|v_d|>2$. Using Lemma~\ref{lem:short1}, we have
\[
G\ge\sigma_d+\sigma_c+\theta\ge 2\sigma_d+\theta= 2(G-2+v_{d,0})+\theta,
\]
thus
\begin{equation}\label{eq:Gbound1}
G\le 4-\theta-2v_{d,0}.
\end{equation}

If $d\le2$, by Corollary~\ref{cor:vd} we have $v_{d,0}=1$ and $d=2$. We have $x_0\in[v_2]$ while $x_0\notin[w]$. Since $|v_2|=3<|w|$, we must have $|\braket{v_2}{w}|\le1$. Then $\theta=1$ by Corollary~\ref{cor:epsilon+1}. So $G\le 1$ by (\ref{eq:Gbound1}), which is not possible.

If $d\ge3$,  it follows from (\ref{eq:Gbound1}) that
\[4-\theta-2v_{d,0}\ge G\ge d+1\ge4,\] so $\theta=-1$, $v_{d,0}=0$, $d\le4$ and $G\le5$.
\end{proof}

Proposition~\ref{prop:Gbound} implies that $d\in\{0,1,2,3, 4\}$.
We now argue that $d=0$.

{\prop\label{prop:d=0} $d=0$.}
\begin{proof}
Suppose that $d=1$. Using Lemma~\ref{lem:vdtight}, we get that $v_1=-e_1+e_0$. We have that $G=2$ and $\sigma_1=1$.
By Corollary~\ref{cor:epsilon+1} and Lemma~\ref{lem:short1}, we get that
\[
2=G\ge \sigma_c+\sigma_1+1\ge 3,
\]
which is a contradiction. 

Suppose that $d=2$. It follows from Proposition~\ref{prop:Gbound} that $|v_2|=2$.
We separate the cases to whether $\sigma_1(=\sigma_2)$ is $1$ or $2$.

First assume that $\sigma_1=\sigma_2=1$.
If $c\ne 3$, then $x_0\in [v_3]$, thus $[v_2]$ and $[v_3]$ share their left end. So $\braket{v_3}{v_2}\ne0$. In particular, $1\not \in \supp(v_3)$. Since $\sigma_0=\sigma_1=1$, $0\not \in \supp(v_3)$, so $|v_3|=2$, which is impossible as $\sigma_3>1$ by (\ref{eq:abcdMAX}). If $c=3$,
note that $\theta=1$ by Corollary~\ref{cor:epsilon+1}, by Lemma~\ref{lem:Gbound} we have
\[
3=G\ge\min\{\sigma_f,\sigma_3+2\}.
\]
By (\ref{eq:abcdMAX}), $\sigma_f>\sigma_3$, so we have $\sigma_3\le2$.
If $\sigma_3=1$, then $\braket{v_3}{w}=0$ and $\braket{v_3}{v_2}=-1$. Since $x_0\not \in [v_3]$, $[v_3]$ abuts the right endpoint of $[v_2]$. Since $[v_2]\dagger [w]$ by Lemma~\ref{lem:wdagger}, we get $\braket{v_3}{w}\not = 0$, a contradiction. If $\sigma_3=2$, then $v_3=-e_3+e_2+e_1$. We have $v_3\sim v_1\sim v_2$, $|v_1|=|v_2|=2$, $[v_2]\dagger [w]$, so $[v_3]$ contains the leftmost high weight vertex of $[w]$, which contradicts the fact that $\braket{v_3}{w}=0$.

Next we suppose that ($d=2$ and) $\sigma_1=\sigma_2=2$. Then $v_1=2e_0-e_1$, $v_2=e_1-e_2$. We have $x_0\in[v_2]$, $x_0\notin[v_1],[w]$, and $[v_2]$ abuts both $[v_1]$ and $[w]$. So $[v_1]$ and $[w]$ share their left endpoint. It follows that $|\braket{v_1}{w}|=4$, which is not possible by Remark~\ref{wpairing}.

Suppose $d\ge3$. Proposition~\ref{prop:Gbound} implies that $v_d=-e_d+e_{d-1}+\cdots+e_1$. Also, since $5\ge G\ge2+\sigma_1+\sigma_2$, we find that $\sigma_1=1$. Consider the vector $v'_d=v_d-e_1+e_0$. Since $\theta=-1$ by Proposition~\ref{prop:Gbound}, $\braket{v'_d}{w}=0$. Using Corollary~\ref{cor:epsilon+1}, we get $\theta=1$, a contradiction.
\end{proof}

\section{{The case $d=0$}}\label{sec:d=0}
We now turn our attention to the classification in the case $d=0$: in what follows, we classify all changemaker linear lattices of this sort.
\begin{lemma}\label{lem:d=0}
$c = 1$, $\sigma_c = 1$, and $\sigma_a = \sigma_b + 1$.
\end{lemma}
\begin{proof}
By Lemma~\ref{lem:Gbound}, we have
 \[
1=G\ge\min\{\sigma_f,\sigma_a-\sigma_b\}\ge\min\{\sigma_f,\sigma_c+\sigma_0-1\}=\min\{\sigma_f,\sigma_c\}.
\]
Using (\ref{eq:abcdMAX}),
we get $\sigma_c=1$, $c=1$, and $\sigma_a=\sigma_b+1$.
\end{proof}

For the rest of the section, we will replace $w$ in~\eqref{eq:w} with
\begin{equation}\label{eq:w'}
w'=-e_a+e_b+e_c.
\end{equation}
The following is an immediate corollary of Lemma~\ref{lem:d=0}.
\begin{cor}
The vector $w'$ is an irreducible, unbreakable vector in $\sigma^\perp$, and $x_0\in[w']$.
\end{cor}

\begin{lemma}
$b = 2$, $\sigma_b = 1$, and $\sigma_a = 2$. Hence $(\sigma_0,\dots,\sigma_a) = (1,1,1,2^{[s]},2)$ for some $s \ge 0$.
\end{lemma}
\begin{proof}
Suppose towards a contradiction that $b > 2$. Since $\sigma_0 = \sigma_1 = 1$ and $b > 2$, $\sigma_2\in\{2,3\}$. 

If $\sigma_2 = 2$, then $\braket{v_2}{v_1} = 0$, $\braket{v_2}{w'} = 1$ and $\braket{v_1}{w'} = -1$. Since $|v_1|=2$ and $x_0\notin[v_1]$, $[v_1]$ abuts the right end of $[w']$. If $[v_2]$ also abuts $[w']$, noting that $x_0\notin [v_2]$,
it abuts the right end of $[w']$, so $[v_2]$ abuts $[v_1]$, contradicting the fact that $\braket{v_2}{v_1} = 0$. Thus we must have $[v_2]\pitchfork[w']$, $|[v_2]\cap[w']|=3$, $v_2=\epsilon[v_2]$ and $w'=\epsilon[w']$ for some $\epsilon\in\{1,-1\}$. 
It follows that $w'-v_2$ is reducible. However, $w'-v_2=-e_a+e_b+e_2-e_0$ is irreducible by Lemma~\ref{lem:Decomp} and the fact that $\sigma_a=\sigma_b+1$, a contradiction.

If $\sigma_2 = 3$, then $[v_2]$ contains $x_0$, so $[w'] \prec [v_2]$. However, since $|w'| = 3$, this can happen only if $|\braket{v_2}{w'}| = 2$, contradicting the fact that $\braket{v_2}{w'} = 1$.

Having proved $b = 2$, we must have $\sigma_2\in\{1,2,3\}$. If $\sigma_2 = 2$, the interval $[v_2]$ contains $x_0$, so $[v_2]$ and $[w']$ share their left end, a contradiction to the direct computation
$\braket{v_2}{w'} = 0$. If $\sigma_2 = 3$,
using Proposition~\ref{prop:T0T1}, there must be some $\chi \in \mathcal{T}_1$ with $\braket{\chi}{\sigma} = 2$. Moreover, since $\{0,1,2\} = \{d,c,b\}$, $\sigma_f>\sigma_2 = 3$. Therefore, $\chi_a \not = 0$ by Proposition~\ref{prop:Sh1}. Using Lemma~\ref{lem:d=0}, $\sigma_a = 4$. It must be the case that for some $i\in \{b,c,d\}$ $\chi_i=-1$ and $\chi_j=0$ for $j\not = i, a$. Then $\braket{\chi}{\sigma}$ is either $1$ or $3$, a contradiction.

Therefore, $b=2$, $\sigma_2 = 1$, and $\sigma_a = \sigma_b + 1 = 2$.
\end{proof}

\begin{lemma}\label{lem:d0class}
$\sigma_i = 2s+3$ for $i > a$. That is, $\sigma = (1,1,1,2^{[s]},2,2s+3^{[t]})$ with $s,t\ge 0$.
\end{lemma}
\begin{proof}
First, consider $v_{a+1}$. Since $\sigma_{a+1} > 2$, $m:=\min \supp(v_{a+1}) < a$, so if $m\ge 3$ then $s:=a-3>0$ and there would be a claw centered at $v_m$, a contradiction to Lemma~\ref{claw}. Therefore, $\supp(v_{a+1}) \cap \{0,1,2\}$ is nonempty, thus is one of $\{0,1,2\}$, $\{1,2\}$, or $\{2\}$ by Lemma~\ref{gappy3}. 

We note that $x_0\in[v_a]$ no matter $s=0$ or $s>0$.

We claim that there is no index $j$ such that $v_j$
 is tight. Otherwise, we have $j>a$ and $[v_j]$ contains $x_0$, so $[v_a] \prec [v_j]$.
If $s > 0$, $\braket{v_a}{v_j} = 0$, a contradiction to $[v_a] \prec [v_j]$. If $s = 0$, then $|v_a|=3$ hence $|\braket{v_a}{v_{j}}| = 2$, contradicting the direct computation  $\braket{v_a}{v_{j}} = 1$.

If $m=0$, then $3 \in \supp(v_{a+1})$ since otherwise $\braket{v_3}{v_{a+1}} = 2$, a contradiction to Lemma~\ref{lem:Pairing1}. Then since $|v_i| = 2$ for $3 < i \le a$, $v_{a+1}$ is just right by the claim in the last paragraph. However, if $s > 0$, then $(v_3;v_4,v_1,v_{a+1})$ will give a claw, a contradiction~(Lemma~\ref{claw}). If $s = 0$ then $[v_3]$ contains $x_0$ so $[v_1]$ and $[v_{a+1}]$ must both abut the right end of $[v_3]$, contradicting the fact that they are orthogonal.

If $m=1$, then again we must have $3 \in \supp(v_{a+1})$ and $v_{a+1}$ just right. Since $|\{a,b,c,d\}\cap\supp(v_{a+1})|=3$, $x_0\in[v_{a+1}]$, so $[v_a] \prec [v_{a+1}]$ and $|\braket{v_{a+1}}{v_a}| = |v_a|-1$. This contradicts the direct computation of $\braket{v_a}{v_{a+1}}$ no matter $s=0$ or $s>0$.

If $m=2$, then $v_{a+1} = e_2 + e_k + \cdots + e_a - e_{a+1}$ for some $3 \le k \le a$. If $3 < k < a$, there is a claw $(v_k; v_{k-1},v_{k+1}, v_{a+1})$~(Lemma~\ref{claw}). If $k = a$ and $a > 3$, then $x_0\in[v_a]$ but $x_0\notin[v_{a+1}]$, and so $[v_a]\dagger[v_{a+1}]$ since $|v_a|=2<|v_{a+1}|$. If $s=1$, then since $x_0\not \in [v_3]$, $\braket{v_{3}}{v_a} =-1$, $[v_3]$ and $[v_{a+1}]$ will share a hight weight vertex, which is not possible. If $s>1$, then both $[v_{a+1}]$ and $[v_{a-1}]$ abut the right endpoint of $[v_a]$, hence $\braket{v_{a+1}}{v_{a-1}}=\pm1$, a contradiction to the direct computation $\braket{v_{a+1}}{v_{a-1}}=0$. Therefore, $k = 3$, so $v_{a+1}$ is just right and $\sigma_{a+1} = 2s+3$.

Finally, suppose that for some $j > a + 1$, $|v_j| > 2$. Take $j$ to be the smallest such index. Then $v_j$ is unbreakable by our earlier claim. 
Let $\ell=\min \supp(v_j)$. If either $\ell \ge a+1$ or $3\le\ell<a$, there will be a claw centered at $v_{\ell}$, contradicting~Lemma~\ref{claw}. If $\ell=a$, then $[v_j]$ contains $x_0$, so $[v_a] \prec [v_j]$.
If $s=0$, $|v_a|=3$, thus $[v_j]$ contains the high weight vertex of $[v_a]$, a contradiction.
If $s>0$, $[v_3]$ is connected to $[v_a]$ via a (possibly empty) sequence  of norm $2$ vectors, so the intervals $[v_3]$ and $[v_j]$ will share a high weight vertex, a contradiction. If $\ell < 3$, then there is a heavy triple $(v_3, v_{a+1}, v_j)$, contradicting Lemma~\ref{heavytriple}. 
\end{proof}

\section{Proof of Theorem~\ref{thm:Realization}}\label{sec:pandq}
Lemma~\ref{lem:d0class} specifies a changemaker vector in $\Z^{n+2}$ whose orthogonal complement is the linear changemaker lattice $\Lambda(q,-p)$. From the integers $a_0, a_1, \cdots a_n$ in~\eqref{eq:ai}, we can recover $p$ and $q$ using~\eqref{eq:ContFrac}. Since $q<p<2q$, we have
\[\frac{p}{q}=[2,a_0,a_1,\dots,a_n]^-.\]
We use the following facts:	
\begin{lemma}\cite[Lemma~9.5~(2)~and~(3)]{greene:LSRP}\label{greene9.5}
For integers $s,t,b$ with $b\ge 2$ and $s,t\ge 0$,
\begin{itemize}
    \item[1.] $[\cdots, b, 2^{[t-1]}]^- = [\cdots, b-1, -t]^-$.
    \item[2.] If $\displaystyle [2^{[s+1]},b,\cdots]^- = \frac{p}{q}$, then $\displaystyle[-(s+2),b-1, \cdots]^-=\frac{p}{q-p}$.
\end{itemize}
\end{lemma}

We have
\[
\sigma = (1,1,1,2^{[s]},2,2s+3^{[t]}),
\]
with $s,t\ge0$. One can check that the standard basis of the linear changemaker lattice
\[
S=\{v_{s+3},\cdots,v_3,v_1,v_2,v_{s+4},\cdots,v_{s+t+3}\}
\]
coincides with its vertex basis with norms given by
\[
\{2^{[s]},3,2,2,s+3,2^{[t-1]}\}.
\]
By Lemma~\ref{lem:OddPairing}, $[v_{s+3}]$ contains $x_0$, so $v_{s+3}=x_0$.
Hence we have
\[\frac{p}{q}=[2^{[s+1]},3,2,2,s+3,2^{[t-1]}]^-.\]
Using Lemma~\ref{greene9.5}, we see that
\begin{align*}
&q=7+4s+9t+12st+4s^2t,\text{ and} \\
&p=11+4s+14t+16st+4s^2t.
\end{align*}
It is straightforward to check that
\[
\displaystyle q=\frac{1}{r^2-2r-1}(r^2p-1),
\]
with $r=-2s-3$ and $p\equiv -2r+5 \pmod{r^2-2r-1}$.

\begin{proof}[Proof of Theorem~\ref{thm:Realization}]
Suppose $P(p,q)\cong S^3_{4q}(K)$, the above computation shows that $(p,q)$ must be as in the statement. On the other hand, if $(p,q)$ is as in the statement, it follows from~\cite[Table~2]{Prism2016} that there exists a Berge--Kang knot $K_0$ such that $P(p,q)\cong S^3_{4q}(K_0)$. For the second statement, we note that $K$ and $K_0$ correspond to the same changemaker vector. Using Proposition~\ref{prop:AlexanderComputation}, we know that $\Delta_K=\Delta_{K_0}$, so $\widehat{HFK}(K)\cong \widehat{HFK}(K_0)$ by \cite[Theorem~1.2]{OSzLens}.
\end{proof}
\bibliographystyle{amsalpha2}
\bibliography{Reference}

\end{document}